\documentclass[oneside]{amsart}
\usepackage[T1]{fontenc}
\usepackage[latin9]{inputenc}
\usepackage{amsthm}
\usepackage{graphicx}
\usepackage{amssymb}
\usepackage[unicode=true, pdfusetitle,
 bookmarks=true,bookmarksnumbered=false,bookmarksopen=false,
 breaklinks=false,pdfborder={0 0 1},backref=false,colorlinks=false]
 {hyperref}

\providecommand{\tabularnewline}{\\}
\numberwithin{equation}{section} %% Comment out for sequentially-numbered
\numberwithin{figure}{section} %% Comment out for sequentially-numbered
\theoremstyle{plain}
\newtheorem{thm}{Theorem}[section]
  \theoremstyle{definition}
  \newtheorem{defn}[thm]{Definition}
  \theoremstyle{plain}
  \newtheorem{algorithm}[thm]{Algorithm}
  \theoremstyle{plain}
  \newtheorem{prop}[thm]{Proposition}
 \theoremstyle{definition}
  \newtheorem{example}[thm]{Example}
  \theoremstyle{plain}
  \newtheorem{assumption}[thm]{Assumption}
  \theoremstyle{plain}
  \newtheorem{lem}[thm]{Lemma}
  \theoremstyle{plain}
  
  \theoremstyle{remark}
  \newtheorem{rem}[thm]{Remark}
  \theoremstyle{definition}
  \newtheorem{problem}[thm]{Problem}
  \theoremstyle{remark}
  \newtheorem*{acknowledgement*}{Acknowledgement}

\newcommand{\lev}{\mbox{\rm lev}}

\newcommand{\intr}{\mbox{\rm int}}
\newcommand{\Range}{\mbox{\rm Range}}
\newcommand{\diam}{\mbox{\rm diam}}
\newcommand{\unitt}{\mbox{\rm unit}}

\begin{document}

\title{Level set methods for finding saddle points of general Morse index}

\author{C.H. Jeffrey Pang}

\curraddr{Department of Combinatorics and Optimization, Mathematics, University
of Waterloo.}

\email{chj2pang@math.uwaterloo.ca}

\date{\today}
\begin{abstract}
For a function $f:X\to\mathbb{R}$, a point is critical if its derivatives
are zero, and a critical point is a saddle point if it is not a local extrema. In this
paper, we study algorithms to find saddle points of general Morse
index. Our approach is motivated by the multidimensional mountain
pass theorem, and extends our earlier work on methods (based on studying
the level sets of $f$) to find saddle points of mountain pass type.
We prove the convergence of our algorithms in the nonsmooth case,
and the local superlinear convergence of another algorithm in the
smooth finite dimensional case.
\end{abstract}



\maketitle
\tableofcontents{}

\section{Introduction}

For a function $f:X\rightarrow\mathbb{R}$, we say that $x$ is a
\emph{critical point} if $\nabla f(x)=\mathbf{0}$, and $y$ is a
\emph{critical value} if there is some critical point $x$ such that
$f(x)=y$. A critical point $x$ is a \emph{saddle point} if it is
neither a local minimizer nor a local maximizer. In this paper, we
present algorithms based on the multidimensional mountain pass theorem
to find saddle points numerically.

The main purpose of critical point theory is the study of variational
problems. These are problems (P) such that there exists a smooth functional
$\Phi:X\rightarrow\mathbb{R}$ whose critical points are solutions
of (P). Variational problems occur frequently in the study of partial
differential equations. 

At this point, we make a remark about saddle points in the study of
min-max problems. Such saddle points occur in problems in game theory
and in constrained optimization using the Lagrangian, and have the
splitting structure\[
\min_{x\in X}\max_{y\in Y}f(x,y).\]
In min-max problems, this splitting structure is exploited in numerical
procedures. See \cite{RH03} for a survey of algorithms for min-max
problems. In the general case, for example in finding weak solutions
of partial differential equations, such a splitting structure may
only be obtained after the saddle point is located, and thus is not
helpful for finding the saddle point.

A critical point $x$ is \emph{nondegenerate} if its Hessian $\nabla^{2}f(x)$
is nonsingular and it is \emph{degenerate} otherwise. The \emph{Morse
index} of a critical point is the maximal dimension of a subspace
of $X$ on which the Hessian $\nabla^{2}f(x)$ is negative definite.
In the finite dimensional case, the Morse index is the number of negative
eigenvalues of the Hessian.

Local maximizers and minimizers of $f:X\rightarrow\mathbb{R}$ are
easily found using optimization, while saddle points are harder to
find. To find saddle points of Morse index 1, one can use algorithms
motivated by the mountain pass theorem. Given points $a,b\in X$ ,
define a \emph{mountain pass} $p^{*}\in\Gamma(a,b)$ to be a minimizer
of the problem \[
\inf_{p\in\Gamma(a,b)}\sup_{0\leq t\leq1}f(p(t)),\]
if it exists. Here, $\Gamma(a,b)$ is the set of continuous paths
$p:[0,1]\rightarrow X$ such that $p(0)=a$ and $p(1)=b$. Ambrosetti
and Rabinowitz's \cite{AR73} mountain pass theorem states that under
added conditions, there is a critical value of at least $\max\{f(a),f(b)\}$.
To find saddle points of higher Morse index, it is instructive to
look at theorems establishing the existence of critical points of
Morse index higher than 1. Rabinowitz \cite{R77} proved the multidimensional
mountain pass theorem which in turn motivated the study of linking
methods to find saddle points. We shall recall theoretical material
relevant for finding saddle points of higher Morse index in this paper
as needed.

While the study of numerical methods for the mountain pass problem
began in the 70's or earlier to study problems in computational chemistry,
Choi and McKenna \cite{CM93} were the first to propose a numerical
method for the mountain pass problem to solve variational problems.
Most numerical methods for finding critical points of mountain pass
type rely on discretizing paths in $\Gamma(a,b)$ and perturbing paths
to lower the maximum value of $f$ on the path. There are a few other
methods of finding saddle points of mountain pass type that do not
involve perturbing paths, for example \cite{H04,BT07}. 

Saddle points of higher Morse index are obtained with modifications
of the mountain pass algorithm. Ding, Costa and Chen \cite{DCC99}
proposed a numerical method for finding critical points of Morse index
2, and Li and Zhou \cite{LZ01} proposed a method for finding critical
points of higher Morse index. 

In \cite{LP08}, we suggested a numerical method for finding saddle
points of mountain pass type. The key observation is that the value
\[
\sup\big\{l\geq\max\big(f(a),f(b)\big)\mid a,b\mbox{ lie in different path components of }\{x\mid f(x)\leq l\}\big\}\]
is a critical value. In other words, the supremum of all levels $l$
such that there is no path connecting $a$ and $b$ in the level set
$\{x\mid f(x)\leq l\}$ is a critical value. See Figure \ref{fig:mtn-contrast}
for an illustration of the difference between the two approaches.
An extensive theoretical analysis and some numerical results of this
approach were provided in \cite{LP08}. 

In this paper, we extend three of the themes in the level set approach
to find saddle points of higher Morse index, namely the convergence
of the basic algorithm (Sections \ref{sec:Alg-desc} and \ref{sec:Conv-ppties}),
optimality condition of sub-problem (Section \ref{sec:Opti-condns}),
and a fast locally convergent method in $\mathbb{R}^{n}$ (Sections
\ref{sec:Fast-local-convergence} and \ref{sec:Proof-of-superlinear}).
Section \ref{sec:Other-local-analysis} presents an alternative result
on convergence to a critical point similar to that of Section \ref{sec:Conv-ppties}.

We refer the reader to \cite{LP08} for examples reflecting the limitations
of the level set approach for finding saddle points of mountain pass
type, which will be relevant for the design of level set methods of
finding saddle points of general Morse index. 

\begin{figure}[h]
\begin{tabular}{|c|c|}
\hline 
\includegraphics[scale=0.4]{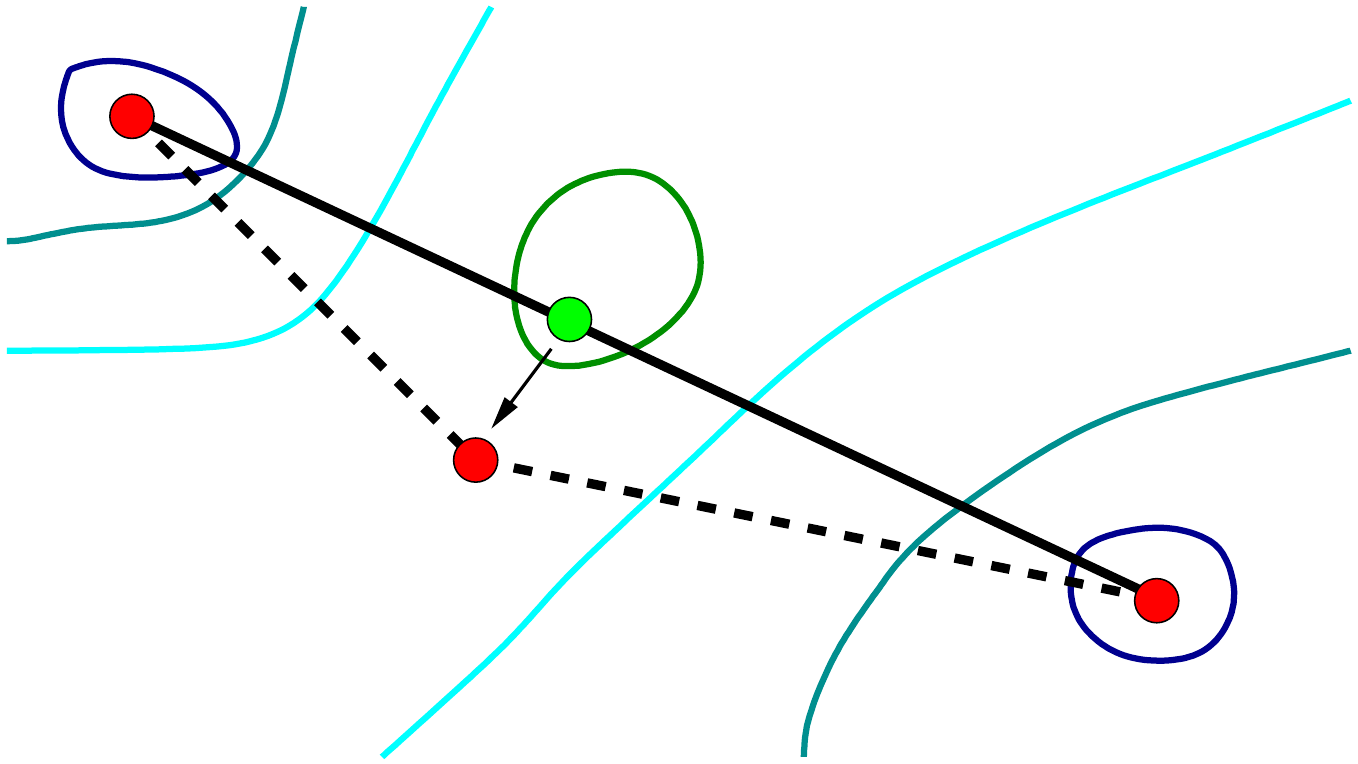} & \includegraphics[scale=0.3]{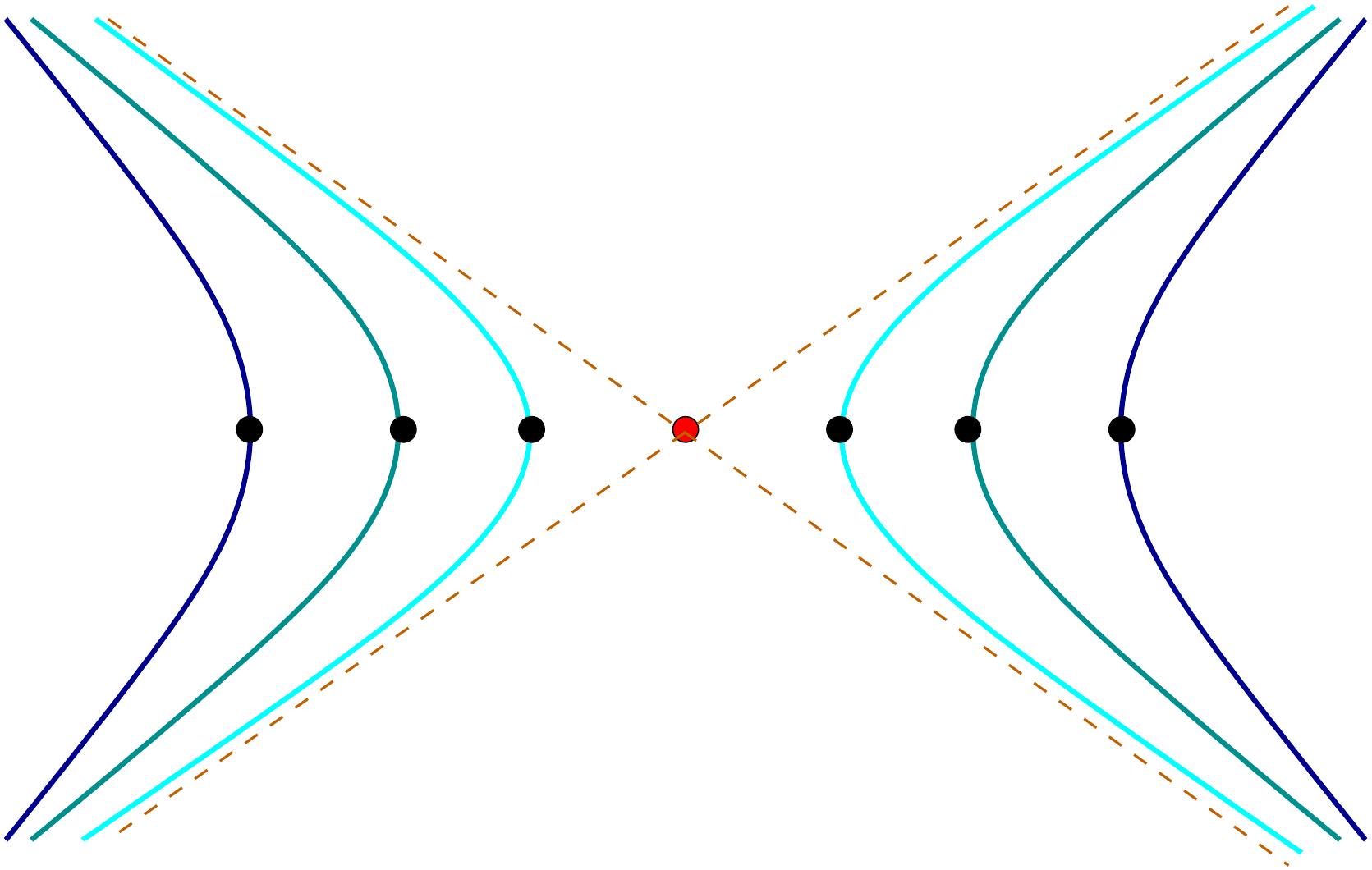}\tabularnewline
\hline
\end{tabular}

\caption{\label{fig:mtn-contrast}The diagram on the left shows the classical
method of perturbing paths for the mountain pass problem, while the
diagram on the right shows convergence to the critical point by looking
at level sets.}

\end{figure}

\section*{Notation}
\begin{description}
\item [{$\lev_{\geq b}f$}] This is the level set $\{x\mid f(x)\geq b\}$,
where $f:X\rightarrow\mathbb{R}$. The interpretations of $\lev_{\leq b}f$
and $\lev_{=b}f$ are similar.
\item [{$\mathbb{B}$}] The ball of center $\mathbf{0}$ and radius $1$.
$\mathbb{B}(x,r)$ stands for a ball of center $x$ and radius $r$.
$\mathbb{B}^{n}$ denotes the $n$-dimensional sphere in $\mathbb{R}^{n}$.
\item [{$\mathbb{S}^{n}$}] The $n$-dimensional sphere in $\mathbb{R}^{n+1}$.
\item [{$\partial$}] Subdifferential of a real-valued function, or the
relative boundary of a set. If $h:\mathbb{B}^{n}\to S$ is a homeomorphism
between $\mathbb{B}^{n}$ and $S$, then the relative boundary of
$S$ is $h(\mathbb{S}^{n-1})$.
\item [{$\mbox{lin}(A)$}] For an affine space $A$, the lineality space
$\mbox{lin}(A)$ is the space $\{a-a^{\prime}\mid a,a^{\prime}\in A\}$.
\end{description}

\section{\label{sec:Alg-desc}Algorithm for critical points}

We look at the critical point existence theorems to give an insight
on our algorithm for finding critical points of higher Morse index
below. Here is the definition of linking sets. We take our definition
from \cite[Section II.8]{Str08}.
\begin{defn}
(Linking) Let $A$ be a subset of $\mathbb{R}^{n}$, $B$ a submanifold
of $\mathbb{R}^{n}$ with relative boundary $\partial B$. Then we
say that $A$ and $\partial B$ \emph{link} if 

(a) $A\cap\partial B=\emptyset$, and 

(b) for any continuous $h:\mathbb{R}^{n}\to\mathbb{R}^{n}$ such that
$h\mid_{\partial B}=id$ we have $h(B)\cap A\neq\emptyset$.
\end{defn}
Figure \ref{fig:Linking-subsets} illustrates two examples of linking
subsets in $\mathbb{R}^{3}$. In the diagram on the left, the set
$A$ is the union of two points inside and outside the sphere $B$.
In the diagram on the right, the sets $A$ and $B$ are the interlocking
'rings'. Note however that $A$ and $B$ link does not imply that
$B$ and $A$ link, though this will be true with additional conditions.
We hope this does not cause confusion. 

We now recall the Palais-Smale condition.
\begin{defn}
(Palais-Smale condition) Let $X$ be a Banach space and $f:X\rightarrow\mathbb{R}$
be $\mathcal{C}^{1}$. We say that a sequence $\{x_{i}\}_{i=1}^{\infty}\subset X$
is a \emph{Palais-Smale sequence} if $\{f(x_{i})\}_{i=1}^{\infty}$
is bounded and $\nabla f(x_{i})\rightarrow\mathbf{0}$, and $f$ satisfies
the \emph{Palais-Smale condition} if any Palais-Smale sequence admits
a convergent subsequence.
\end{defn}
The classical multidimensional pass theorem originally due to Rabinowitz
\cite{R77} states that under added conditions, if there are linking
sets $A$ and $B$ such that $\max_{A}f<\min_{B}f$ and the Palais-Smale
condition holds, then there is a critical value of at least $\max_{A}f$
for the case when $f$ is smooth. (See Theorem \ref{thm:Rabinowitz-77}
for a statement of the multidimensional mountain pass theorem) Generalizations
in the nonsmooth case are also well-known in the literature. See for
example \cite{Jab03}.

To find saddle points of Morse index $m$, we consider finding a sequence
of linking sets $\{A_{i}\}_{i=1}^{\infty}$ and $\{B_{i}\}_{i=1}^{\infty}$
such that $\mbox{diam}(A_{i})$, the diameter of the set $A_{i}$,
decreases to zero, and the set $A_{i}$ is a subset of an $m$-dimensional
affine space. This motivates the following algorithm.

\begin{figure}
\begin{tabular}{|c|c|}
\hline 
\includegraphics[scale=0.5]{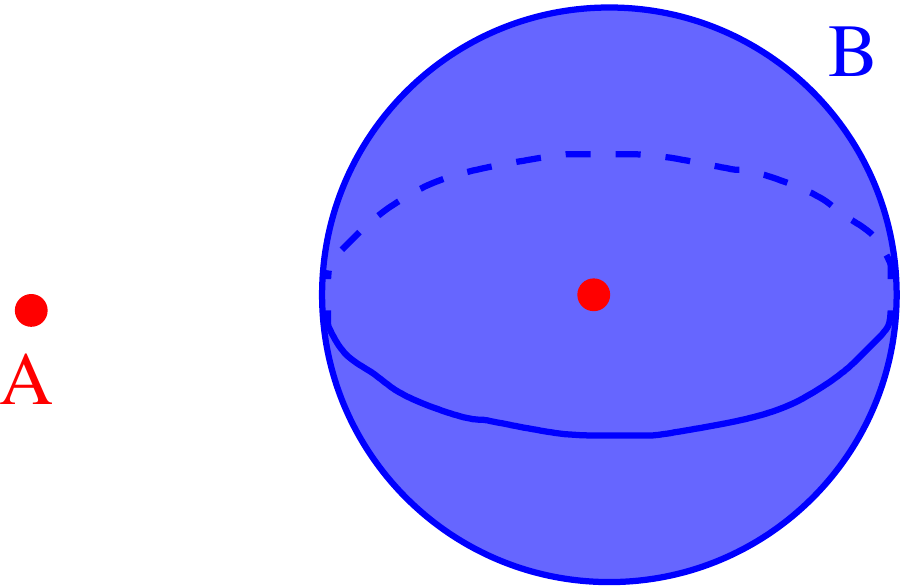} & \includegraphics[scale=0.5]{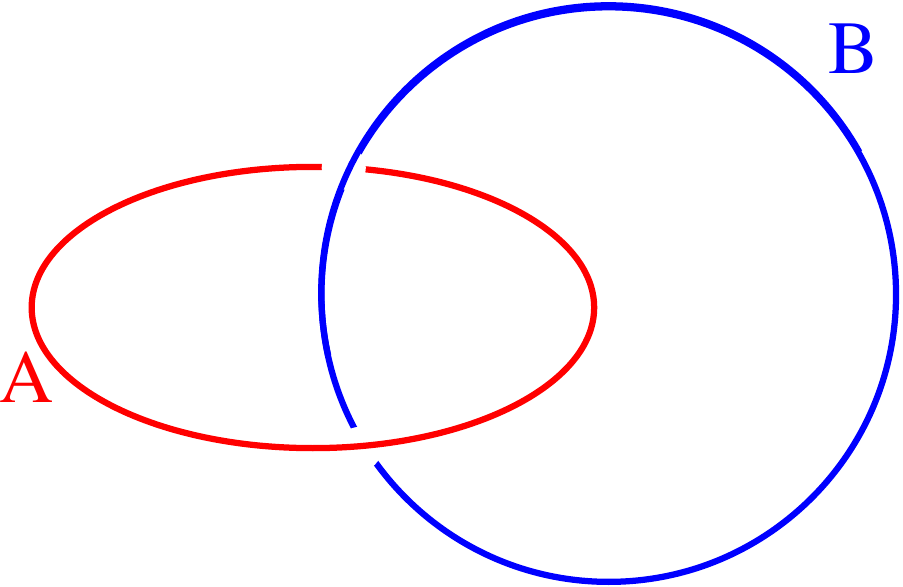}\tabularnewline
\hline
\end{tabular}

\caption{\label{fig:Linking-subsets}Linking subsets}

\end{figure}
 
\begin{algorithm}
\label{alg:saddle-points-outer}First algorithm for finding saddle
points of Morse index $m\geq1$.
\begin{enumerate}
\item Set the iteration count $i$ to $0$, and let $l_{i}$ be a lower
bound of the critical value and $u_{i}$ be an upper bound.
\item Find $x_{i}$ and $y_{i}$, where $(S_{i},x_{i},y_{i})$ is an optimizing
triple of \begin{equation}
\min_{S\in\mathcal{S}}\max_{x,y\in S\cap(\scriptsize\lev_{\geq\frac{1}{2}(l_{i}+u_{i})}f)\cap U_{i}}|x-y|,\label{eq:keyexp1}\end{equation}
where $U_{i}$ is some open set. Here, $\mathcal{S}$ is the set of
$m$-dimensional affine subspaces of $\mathbb{R}^{n}$ intersecting
$U_{i}$. In the inner maximum problem above, we take the value to
be $0$ if $S\cap(\lev_{\leq\frac{1}{2}(l_{i}+u_{i})}f)\cap U_{i}$
is empty, making the objective function above equal to $0$. For simplicity,
we shall just assume that minimizers and maximizers of the above problem
exist.
\item (Bisection) If the objective of \eqref{eq:keyexp1} is zero, then
$\frac{1}{2}(l_{i}+u_{i})$ is a lower bound of the critical value.
Set $l_{i+1}=\frac{1}{2}(l_{i}+u_{i})$ and $u_{i+1}=u_{i}$. Otherwise,
set $l_{i+1}=l_{i}$ and $u_{i+1}=\frac{1}{2}(l_{i}+u_{i})$. 
\item Increase $i$ and go back to step 2.
\end{enumerate}
\end{algorithm}
The critical step of Algorithm \ref{alg:saddle-points-outer} lies
in step 2. We elaborate on optimal conditions that will be a useful
approximate for this step in Section \ref{sec:Opti-condns}. One may
think of the set $A_{i}$ as the relative boundary (to the affine
space $S_{i}$) of $S_{i}\cap(\lev_{\geq l_{i}}f)\cap U_{i}$. A
frequent assumption we will make is nondegenericity.
\begin{defn}
We say that a critical point is \emph{nondegenerate} if its Hessian
is invertible.
\end{defn}
Algorithm \ref{alg:saddle-points-outer} requires $m>0$, but when
$m=0$, nondegenerate critical points of Morse index zero are just
strict local minimizers that can be easily found by optimization.
We illustrate two special cases of Algorithm \ref{alg:saddle-points-outer}.
\begin{example}
(Particular cases of Algorithm \ref{alg:saddle-points-outer}) (a)
For the case $m=1$, $\mathcal{S}$ is the set of lines. The inner
maximization problem in \eqref{eq:keyexp1} has its solution on the
two endpoints of $S_{i}\cap(\lev_{\geq\frac{1}{2}(l_{i}+u_{i})}f)\cap U_{i}$.
This means that \eqref{eq:keyexp1} is equivalent to finding the local
closest points between two components of $(\lev_{\leq\frac{1}{2}(l_{i}+u_{i})}f)\cap U_{i}$,
as was analyzed in \cite{LP08}.

(b) For the case $m=n$, $\mathcal{S}$ contains the whole of $\mathbb{R}^{n}$.
Hence the outer minimization problem in \eqref{eq:keyexp1} is superfluous.
The level set $(\lev_{\geq\frac{1}{2}(l_{i}+u_{i})}f)\cap U_{i}$
gets smaller and smaller as $\frac{1}{2}(l_{i}+u_{i})$ approaches
the maximum value, till it becomes a single point if the maximizer
is unique.
\end{example}

\section{\label{sec:Conv-ppties}Convergence properties}

In this section, we prove the convergence of $x_{i}$, $y_{i}$ in
Algorithm \ref{alg:saddle-points-outer} to a critical point when
they converge to a common limit. We recall some facts about nonsmooth
analysis needed for the rest of the paper. It is more economical to
prove our result for nonsmooth critical points because the proofs
are not that much harder, and nonsmooth critical points are also of
interest in applications.

Let $X$ be a Banach space, and $f:X\rightarrow\mathbb{R}$ be a locally
Lipschitz function at a given point $x$.
\begin{defn}
(Clarke subdifferential) \cite[Section 2.1]{Cla83} Suppose $f:X\rightarrow\mathbb{R}$
is locally Lipschitz at $x$. The \emph{Clarke generalized directional
derivative} of $f$ at $x$ in the direction $v\in X$ is defined
by\[
f^{\circ}(x;v)=\limsup_{t\searrow0,y\rightarrow x}\frac{f(y+tv)-f(y)}{t},\]
where $y\in X$ and $t$ is a positive scalar. The \emph{Clarke subdifferential}
of $f$ at $x$, denoted by $\partial_{C}f(x)$, is the subset of
the dual space $X^{*}$ given by\[
\left\{ \zeta\in X^{*}\mid f^{\circ}(x;v)\geq\left\langle \zeta,v\right\rangle \mbox{ for all }v\in X\right\} .\]
The point $x$ is a \emph{Clarke (nonsmooth) critical point} if $\mathbf{0}\in\partial_{C}f(x)$.
Here, $\left\langle \cdot,\cdot\right\rangle :X^{*}\times X\rightarrow\mathbb{R}$
defined by $\left\langle \zeta,v\right\rangle :=\zeta(v)$ is the
dual relation.
\end{defn}
For the particular case of $\mathcal{C}^{1}$ functions, $\partial_{C}f(x)=\{\nabla f(x)\}$.
Therefore critical points of smooth functions are also nonsmooth critical
points. From the definitions above, it is clear that an equivalent
definition of a nonsmooth critical point is $f^{\circ}(x;v)\geq0$
for all $v\in X$. This property allows us to prove that a point is
nonsmooth critical without appealing to the dual space $X^{*}$.

We now prove our result of convergence to nonsmooth critical points.
\begin{prop}
\label{pro:triples-conv}(Convergence to saddle point) Let $\bar{z}\in X$.
Suppose there is a ball $\mathbb{B}(\bar{z},r)$, a sequence of triples
$\{(S_{i},x_{i},y_{i})\}_{i=1}^{\infty}$ and a sequence $l_{i}$
monotonically increasing to $f(\bar{z})$ such that $(x_{i},y_{i})\to(\bar{z},\bar{z})$
and $(S_{i},x_{i},y_{i})$ is an optimizing triple of \eqref{eq:keyexp1}
in Algorithm \ref{alg:saddle-points-outer} for $l_{i}$ with $U_{i}=\mathbb{B}(\bar{z},r)$.
Then $\bar{z}$ is a Clarke critical point.\end{prop}
\begin{proof}
Seeking a contradiction, suppose there exists some direction $\bar{v}$
such that $f^{\circ}(\bar{z};\bar{v})<0$. This means that there is
some $\bar{\epsilon}>0$ such that if $|z-\bar{z}|<\bar{\epsilon}$
and $\epsilon<\bar{\epsilon}$, then \begin{eqnarray*}
\frac{f(z+\epsilon\bar{v})-f(z)}{\epsilon} & < & \frac{1}{2}f^{\circ}(\bar{z};\bar{v})\\
\Rightarrow f(z+\epsilon\bar{v}) & < & f(z)+\epsilon\frac{1}{2}f^{\circ}(\bar{z};\bar{v}).\end{eqnarray*}
Suppose $i$ is large enough so that $x_{i},y_{i}\in\mathbb{B}(\bar{z},\frac{\bar{\epsilon}}{2})$,
and that $x_{i},y_{i}\in A_{i}:=S_{i}\cap(\lev_{\geq l_{i}}f)\cap\mathbb{B}(\bar{z},r)$
are such that $|x_{i}-y_{i}|=\diam(A_{i})$. Consider the set $\tilde{A}:=(S_{i}+\epsilon_{1}\bar{v})\cap(\lev_{\geq l_{i}}f)\cap\mathbb{B}(\bar{z},r)$,
where $\epsilon_{1}>0$ is arbitrarily small. Let $\tilde{x}_{i},\tilde{y}_{i}\in\tilde{A}$
be such that $|\tilde{x}_{i}-\tilde{y}_{i}|=\diam(\tilde{A})$. From
the minimality of the outer minimization, we have $|\tilde{x}_{i}-\tilde{y}_{i}|\geq|x_{i}-y_{i}|$.
Note that $f(\tilde{x}_{i})=f(\tilde{y}_{i})=l_{i}$. Then \begin{eqnarray*}
f(\tilde{x}_{i}) & < & f(\tilde{x}_{i}-\epsilon_{1}\bar{v})+\epsilon_{1}\frac{1}{2}f^{\circ}(\bar{z};\bar{v})\\
\implies f(\tilde{x}_{i}-\epsilon_{1}\bar{v}) & > & f(\tilde{x}_{i})-\epsilon_{1}\frac{1}{2}f^{\circ}(\bar{z};\bar{v})\\
 & > & l_{i}.\end{eqnarray*}
The continuity of $f$ implies that we can find some $\epsilon_{2}>0$
such that $\hat{x}_{i}:=\tilde{x}_{i}-\epsilon_{1}\bar{v}+\epsilon_{2}(\tilde{x}_{i}-\tilde{y}_{i})$
lies in $A_{i}$. Similarly, $\hat{y}_{i}:=\tilde{y}_{i}-\epsilon_{1}\bar{v}$
lie in $A_{i}$ as well. But\begin{eqnarray*}
|\hat{x}_{i}-\hat{y}_{i}| & > & |\tilde{x}_{i}-\tilde{y}_{i}|\\
 & \geq & |x_{i}-y_{i}|.\end{eqnarray*}
This contradicts the maximality of $|x_{i}-y_{i}|$ in $A_{i}$, and
thus $\bar{z}$ must be a critical point. 
\end{proof}

\section{\label{sec:Opti-condns}Optimality conditions}

We now reduce the min-max problem \eqref{eq:keyexp1} to a condition
on the gradients $\nabla f(x_{i})$ and $\nabla f(y_{i})$ that is
easy to verify numerically. This condition will help in the numerical
solution of \eqref{eq:keyexp1}. We use methods in sensitivity analysis
of optimization problems (as is done in \cite{BS00}) to study how
varying the $m$-dimensional affine space $S$ in an $(m+1)$-dimensional
subspace affects the optimal value in the inner maximization problem
in \eqref{eq:keyexp1}. We conform as much as possible to the notation
in \cite{BS00} throughout this section. 

Consider the following parametric optimization problem $(P_{u})$
in terms of $u\in\mathbb{R}$ as an $m+1$ dimensional model in $\mathbb{R}^{m+1}$
of the inner maximization problem in \eqref{eq:keyexp1}:\begin{eqnarray}
(P_{u}):\qquad v(u):= & \min & F(x,y,u):=-|x-y|^{2}\nonumber \\
 & \mbox{s.t.} & G(x,y,u)\in K,\nonumber \\
 &  & x,y\in\mathbb{R}^{m+1},\label{eq:stylized-problem}\end{eqnarray}
where $G:(\mathbb{R}^{m+1})^{2}\times\mathbb{R}\rightarrow\mathbb{R}^{4}$
and $K\subset\mathbb{R}^{4}$ are defined by \[
G(x,y,u):=\left(\begin{array}{c}
-f(x)+b\\
-f(y)+b\\
(0,0,\dots,0,u,1)x\\
(0,0,\dots,0,u,1)y\end{array}\right),\qquad K:=\mathbb{R}_{-}^{2}\times\{0\}^{2}.\]
The problem $(P_{u})$ reflects the inner maximization problem of
\eqref{eq:keyexp1}. Due to the standard practice of writing optimization
problems as minimization problems, \eqref{eq:stylized-problem} is
a minimization problem instead. We hope this does not cause confusion. 

Let $S(u)$ be the $m$-dimensional subspace orthogonal to $(0,\dots,0,u,1)$.
The first two components of $G(x,y,u)$ model the constraints $f(x)\geq b$
and $f(y)\geq b$, while the last two components enforce $x,y\in S(u)$.
Denote an optimal solution to $(P_{u})$ to be $(\bar{x}(u),\bar{y}(u))$,
and let $(\bar{x},\bar{y}):=(\bar{x}(0),\bar{y}(0))$. We make the
following assumption throughout.
\begin{assumption}
\label{ass:uniqueness}(Uniqueness of optimizers) $(P_{0})$ has a
unique solution $\bar{x}=\mathbf{0}$ and $\bar{y}=(0,\dots,0,1,0)$
at $u=0$. 
\end{assumption}
We shall investigate how the set of minimizers of $(P_{u})$ behaves
with respect to $u$ at $0$.

The derivatives of $F$ and $G$ with respect to $x$ and $y$, denoted
by $D_{x,y}F$ and $D_{x,y}G$, are\begin{eqnarray}
D_{x,y}F(x,y,u) & = & 2\big(\begin{array}{cc}
(y-x)^{T} & (x-y)^{T}\end{array}\big),\nonumber \\
\mbox{ and }D_{x,y}G(x,y,u) & = & \left(\begin{array}{cc}
-\nabla f(x)^{T}\\
 & -\nabla f(y)^{T}\\
(0,0,\dots,0,u,1)\\
 & (0,0,\dots,0,u,1)\end{array}\right),\label{eq:D-x-y}\end{eqnarray}
where the blank terms in $D_{x,y}G(x,y,u)$ are all zero. 

The \emph{Lagrangian} is the function $L:\mathbb{R}^{m+1}\times\mathbb{R}^{m+1}\times\mathbb{R}^{4}\times\mathbb{R}\rightarrow\mathbb{R}$
defined by \[
L(x,y,\lambda,u):=F(x,y,u)+\sum_{i=1}^{4}\lambda_{i}G_{i}(x,y,u).\]
We say that $\lambda:=(\lambda_{1},\lambda_{2},\lambda_{3},\lambda_{4})$,
depending on $u$, is a \emph{Lagrange multiplier }if $D_{x,y}L(x,y,\lambda,u)=\mathbf{0}$
and $\lambda\in N_{K}(G(x,y,u))$, and the set of all Lagrange multipliers
is denoted by $\Lambda(x,y,u)$. Here, $N_{K}(G(x,y,u))$ stands for
the \emph{normal cone} defined by \[
N_{K}\big(G(x,y,u)\big):=\{v\in\mathbb{R}^{4}\mid v^{T}[w-G(x,y,u)]\leq0\mbox{ for all }w\in K\}.\]
We are interested in the set $\Lambda(\bar{x},\bar{y},0)$. It is
clear that optimal solutions must satisfy $G(\bar{x},\bar{y},0)=\mathbf{0}$,
so $\lambda\in N_{K}(\mathbf{0})=\mathbb{R}_{+}^{2}\times\mathbb{R}^{2}$. 

The condition $D_{x,y}L(\bar{x},\bar{y},\lambda,0)=\mathbf{0}$ reduces
to \begin{eqnarray*}
D_{x,y}\left(F(x,y,0)+\sum_{i=1}^{4}\lambda_{i}G_{i}(x,y,0)\right)\mid_{x=\bar{x},y=\bar{y}} & = & \mathbf{0}\\
\Rightarrow2\left({\bar{y}-\bar{x}\atop \bar{x}-\bar{y}}\right)+\lambda_{1}\left({-\nabla f(\bar{x})\atop \mathbf{0}}\right)+\lambda_{2}\left({\mathbf{0}\atop -\nabla f(\bar{y})}\right)\qquad\qquad\qquad\\
+\lambda_{3}\left({(0,0,\dots,0,0,1)^{T}\atop \mathbf{0}}\right)+\lambda_{4}\left({\mathbf{0}\atop (0,0,\dots,0,0,1)^{T}}\right) & = & \mathbf{0}.\end{eqnarray*}

Here, $G_{i}(\bar{x},\bar{y},0)$ is the $i$th row of $G(\bar{x},\bar{y},0)$
for $1\leq i\leq4$. This is exactly the KKT conditions, and can be
rewritten as\begin{eqnarray}
2(\bar{y}-\bar{x})-\lambda_{1}\nabla f(\bar{x})+\lambda_{3}(0,0,\dots,0,0,1)^{T} & = & 0,\nonumber \\
2(\bar{x}-\bar{y})-\lambda_{2}\nabla f(\bar{y})+\lambda_{4}(0,0,\dots,0,0,1)^{T} & = & 0.\label{eq:small-KKT}\end{eqnarray}

It is clear that $\lambda_{1}$ and $\lambda_{2}$ cannot be zero,
and so we have \begin{eqnarray*}
\nabla f(\bar{x})^{T} & = & \left(0,0,\dots,0,\frac{2}{\lambda_{1}},\frac{\lambda_{3}}{\lambda_{1}}\right),\\
\nabla f(\bar{y})^{T} & = & \left(0,0,\dots,0,-\frac{2}{\lambda_{2}},\frac{\lambda_{4}}{\lambda_{2}}\right).\end{eqnarray*}
Recall that $\lambda_{1},\lambda_{2}\geq0$, so this gives more information
about $\nabla f(\bar{x})$ and $\nabla f(\bar{y})$.

We next discuss the optimality of the outer minimization problem of
\eqref{eq:keyexp1}, which can be studied by perturbations in the
parameter $u$ of \eqref{eq:stylized-problem}, but we first recall
a result on the first order sensitivity of optimal solutions.
\begin{defn}
(Robinson's constraint qualification) (from \cite[Definition 2.86]{BS00})
We say that \emph{Robinson's constraint qualification }holds at $(\bar{x},\bar{y})\in\mathbb{R}^{m+1}\times\mathbb{R}^{m+1}$
 if the regularity condition \[
\mathbf{0}\in\intr\left\{ G(\bar{x},\bar{y},0)+\Range\big(D_{x,y}G(\bar{x},\bar{y},0)\big)-K\right\} \]
is satisfied.\end{defn}
\begin{thm}
\label{thm:from-BS-4.26}(Parametric optimization) (from \cite[Theorem 4.26]{BS00})
For problem\eqref{eq:stylized-problem}, let $(\bar{x}(u),\bar{y}(u))$
be as defined earlier. Suppose that 
\begin{enumerate}
\item [(i)] Robinson's constraint qualification holds at $(\bar{x}(0),\bar{y}(0))$,
and
\item [(ii)] if $u_{n}\rightarrow0$, then $(P_{u_{n}})$ possesses an
optimal solution $(\bar{x}(u_{n}),\bar{y}(u_{n}))$ that has a limit
point $(\bar{x},\bar{y})$.
\end{enumerate}
Then $v(\cdot)$ is directionally differentiable at $u=0$ and \[
v^{\prime}(0)=D_{u}L(x,y,\lambda,0).\]

\end{thm}
We proceed to prove our result.
\begin{prop}
(Optimality condition on $\nabla f(\bar{y})$) \label{pro:perturb-decrease}Consider
the setup so far in this section and suppose Assumption \ref{ass:uniqueness}
holds. If $\nabla f(\bar{y})$ is not a positive multiple of $(0,0,\dots,0,1,0)^{T}$
at $u=0$, then we can perturb $u$ so that \eqref{eq:stylized-problem}
has an increase in objective. \end{prop}
\begin{proof}
We first obtain first order sensitivity information from Theorem \ref{thm:from-BS-4.26}.
Recall that by definition, Robinson's constraint qualification holds
at $(\bar{x},\bar{y})$ if \[
\mathbf{0}\in\mbox{int}\left\{ G(\bar{x},\bar{y},0)+\mbox{Range}\big(D_{x,y}G(\bar{x},\bar{y},0)\big)-K\right\} .\]
From \eqref{eq:small-KKT}, it is clear that $\nabla f(\bar{x})$
and $(0,\dots,0,0,1)$ are linearly independent, and so are $\nabla f(\bar{y})$
and $(0,\dots,0,0,1)$. From the formula of $D_{x,y}G(\bar{x},\bar{y},0)$
in \eqref{eq:D-x-y}, we see immediately that $\mbox{Range}\left(D_{x,y}G(\bar{x},\bar{y},0)\right)=\mathbb{R}^{4}$,
thus the Robinson's constraint qualification indeed holds.

Suppose that $\lim_{n\to\infty}t_{n}=0$. We prove that part (ii)
of Theorem \ref{thm:from-BS-4.26} holds by proving that $(\bar{x}(t_{n}),\bar{y}(t_{n}))$
cannot have any other limit points. Suppose that $(x^{\prime},y^{\prime})$
is a limit point of $\{(\bar{x}(t_{n}),\bar{y}(t_{n}))\}_{n=1}^{\infty}$.
It is clear that $x^{\prime},y^{\prime}\in S(0)$.

We can find $y_{n}\rightarrow\bar{y}$ such that $y_{n}\in S(t_{n})$
and $f(y_{n})=b$. For example, we can use the Implicit Function Theorem
with the constraints \begin{eqnarray*}
f(y) & = & b,\\
g(y,u) & = & 0,\end{eqnarray*}
where $g(y,u)=(0,0,\dots,0,u,1)^{T}y$. The derivatives with respect
to $y_{m}$ and $y_{m+1}$ are \[
\begin{array}{ccc}
\frac{\partial}{\partial y_{m}}f(\bar{y})=-\frac{2}{\lambda_{2}}, &  & \frac{\partial}{\partial y_{m+1}}f(\bar{y})=\frac{\lambda_{4}}{\lambda_{2}},\\
\frac{\partial}{\partial y_{m}}g(\bar{y},0)=0, &  & \frac{\partial}{\partial y_{m+1}}g(\bar{y},0)=1.\end{array}\]
Therefore, for $y_{1}=y_{2}=\cdots=y_{m-2}=y_{m-1}=0$ and any choice
of $u$ close to zero, there is some $y_{m}$ and $y_{m+1}$ such
that $y\in S(u)$ and $f(y)=b$.

Clearly $|\bar{x}-y_{n}|\leq|\bar{x}(t_{n})-\bar{y}(t_{n})|$. Taking
limits as $n\rightarrow\infty$, we have $|\bar{x}-\bar{y}|\leq|x^{\prime}-y^{\prime}|$.
Since $(\bar{x},\bar{y})$ minimize $F$, it follows that $|\bar{x}-\bar{y}|=|x^{\prime}-y^{\prime}|$,
and by the uniqueness of solutions to $(P_{0})$, we can assume that
$x^{\prime}=\bar{x}$ and $y^{\prime}=\bar{y}$.

Theorem \ref{thm:from-BS-4.26} implies that $v^{\prime}(0)=D_{u}L(x,y,\lambda,0)$.
We now calculate $D_{u}L(x,y,\lambda,0)$. It is clear that $D_{u}G(x,y,\lambda,0)=(0,0,0,1)^{T}$,
and so $D_{u}L(x,y,\lambda,0)=\lambda_{4}$. Since $\nabla f(\bar{y})$
is not a multiple of $(0,0,\dots,0,1,0)^{T}$ at $u=0$, $\lambda_{4}\neq0$,
and this gives the conclusion we need.
\end{proof}
A direct consequence of Proposition \ref{pro:perturb-decrease} is
the following easily checkable condition.
\begin{thm}
(Gradients are opposite) \label{pro:opposite-directions}Let $(S_{i},x_{i},y_{i})$
be an optimizing triple to \eqref{eq:keyexp1} for some $l_{i}$ such
that $S_{i}\cap(\lev_{\geq l_{i}}f)\cap U_{i}$ is closed, and $(x_{i},y_{i})$
is the unique pair of points in $S_{i}\cap(\lev_{\geq l_{i}}f)\cap U_{i}$
satisfying $|x_{i}-y_{i}|=\diam(S_{i}\cap(\lev_{\geq l_{i}}f)\cap U_{i})$.
Then $\nabla f(x_{i})$ and $\nabla f(y_{i})$ are nonzero and point
in opposite directions.\end{thm}
\begin{proof}
We can look at an $m+1$ dimensional subspace which reduces to the
setting that we are considering so far in this section. By Proposition
\ref{pro:perturb-decrease}, $\nabla f(y_{i})$ is a positive multiple
of $x_{i}-y_{i}$ at optimality. Similarly, $\nabla f(x_{i})$ is
a positive multiple of $y_{i}-x_{i}$ at optimality, and the result
follows.
\end{proof}

We remark on how to start the algorithm. We look at critical points
of Morse index 1 first. In this case, two local minima $\bar{x}_{1}$,
$\bar{x}_{2}$ are needed before the mountain pass algorithm can guarantee
the existence of a critical point $\bar{x}_{3}$. For any value above
the critical value corresponding to the critical point of Morse index
1, the level set contains a path connecting $\bar{x}_{1}$ and $\bar{x}_{2}$
passing through $\bar{x}_{3}$. 

To find the next critical point of Morse index 2 we remark that under
mild conditions, if $\lev_{\leq a}f$ contains a closed path homeomorphic
to $\mathbb{S}_{1}$, the boundary of the disc of dimension 2, then
the linking principle guarantees the existence of a critical point
through the multidimensional mountain pass theorem. Theorem \ref{thm:Rabinowitz-77}
which we quote later gives an idea how this is possible. We refer
the reader to \cite{Sch99} and \cite[Chapter 19]{Jab03} for more
details on linking methods. 

We now illustrate with an example that without the assumption that
$(x_{i},y_{i})$ is the unique pair of points satisfying $|x_{i}-y_{i}|=\diam(S_{i}\cap(\lev_{\geq l_{i}}f)\cap U_{i})$,
the conclusion in Theorem \ref{pro:opposite-directions} need not
hold.

\begin{figure}[h]
\begin{tabular}{|c|c|}
\hline 
\includegraphics[scale=0.6]{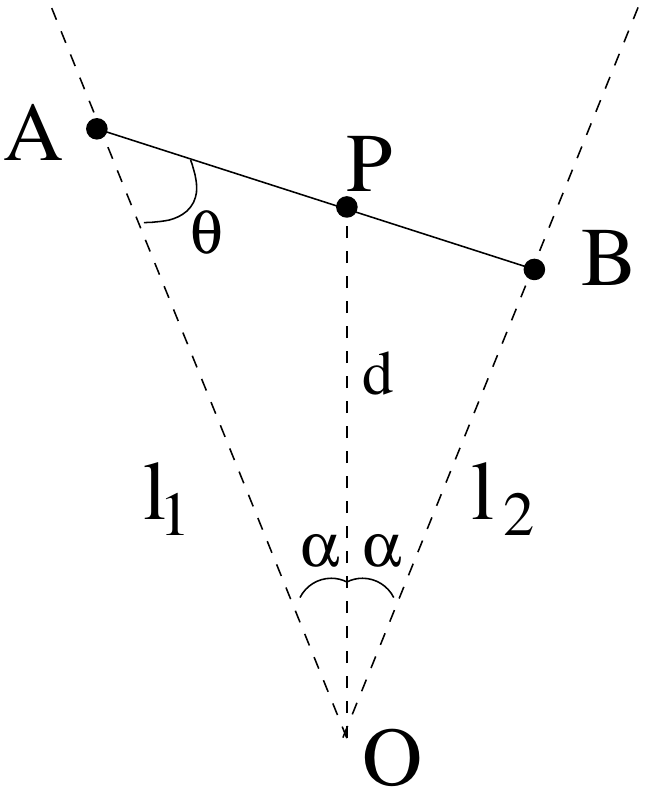} & \includegraphics[scale=0.4]{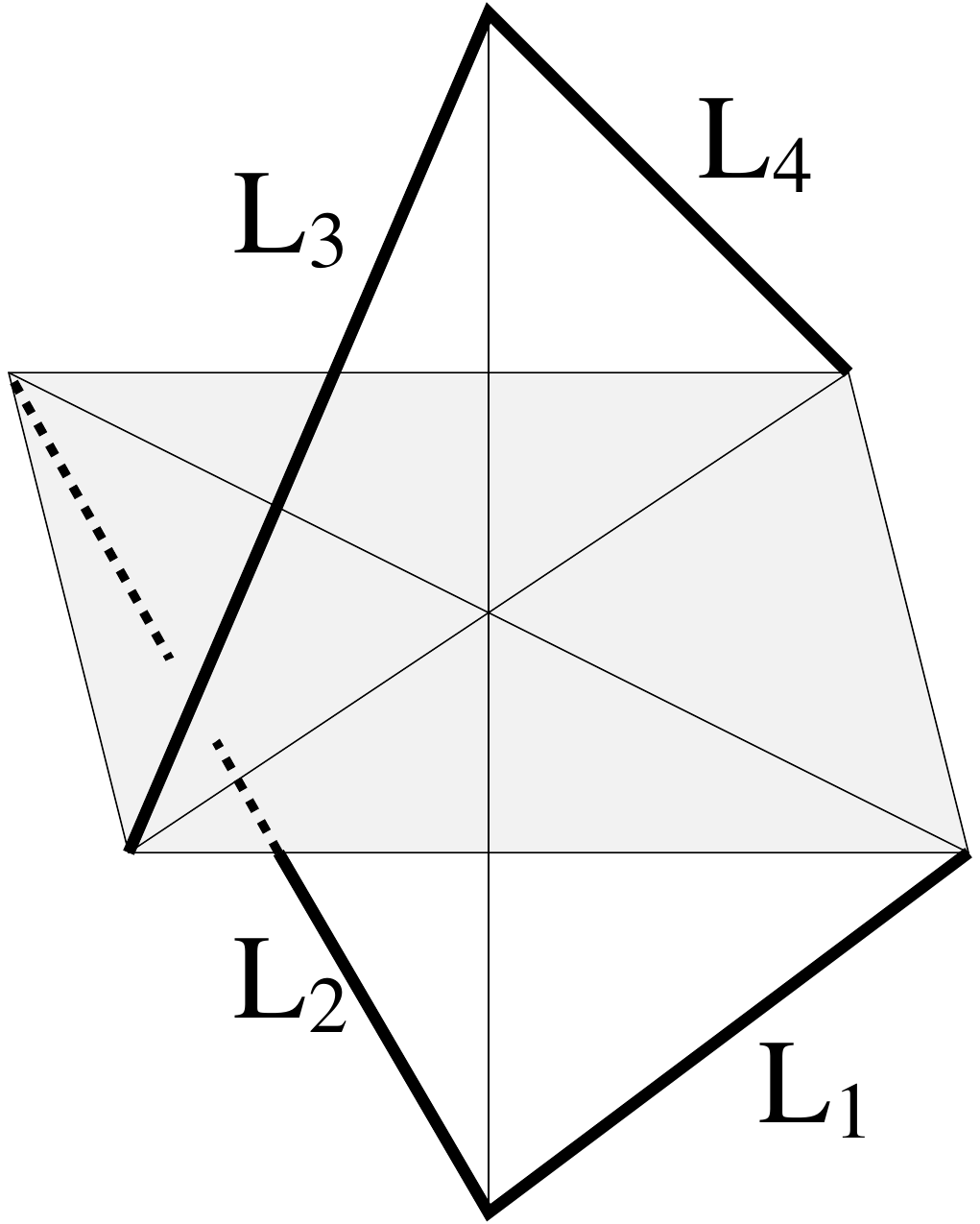}\tabularnewline
\hline
\end{tabular}

\caption{\label{fig:The-lines}The diagram on the left illustrates the setting
of Lemma \ref{lem:intersecting-rays}, while the diagram on the right
illustrates Example \ref{exa:4-lines}.}

\end{figure}

\begin{lem}
(Shortest line segments) \label{lem:intersecting-rays}Suppose lines
$l_{1}$ and $l_{2}$ intersect at the origin in $\mathbb{R}^{2}$,
and let $P$ be a point on the angle bisector as shown in the diagram
on the left of Figure \ref{fig:The-lines}. The minimum distance of
the line segment $AB$, where $A$ is a point on $l_{1}$ and $B$
is a point on $l_{2}$ and $AB$ passes through $P$, is attained
when $OAB$ is an isosceles triangle with $AB$ as its base. \end{lem}
\begin{proof}
Much of this is high school trigonometry and plane geometry, but we
present full details for completeness. Let $\alpha$ be the angle
$\measuredangle AOP$, $\beta$ be the angle $\measuredangle PAO$,
and $d=|OP|$. By using the sine rule, we get \[
|AB|=d\left(\frac{\sin\alpha}{\sin\theta}+\frac{\sin\alpha}{\sin(\pi-2\alpha-\theta)}\right).\]
The problem is now reduced to finding the $\theta$ that minimizes
the value above. Continuing the arithmetic gives:\begin{eqnarray*}
d\left(\frac{\sin\alpha}{\sin\theta}+\frac{\sin\alpha}{\sin(\pi-2\alpha-\theta)}\right) & = & d\sin\alpha\left(\frac{1}{\sin\theta}+\frac{1}{\sin(2\alpha+\theta)}\right)\\
 & = & d\sin\alpha\left(\frac{\sin\theta+\sin(2\alpha+\theta)}{\sin(\theta)\sin(2\alpha+\theta)}\right)\\
 & = & d\sin\alpha\left(\frac{\sin\theta+\sin(2\alpha+\theta)}{\sin(\theta)\sin(2\alpha+\theta)}\right)\\
 & = & d\sin\alpha\left(\frac{2\sin(\alpha+\theta)\cos\alpha}{\frac{1}{2}[\cos(2\alpha)-\cos(2\alpha+2\theta)]}\right)\\
 & = & 2d\sin(2\alpha)\left(\frac{\sin(\alpha+\theta)}{\cos(2\alpha)-\cos(2\alpha+2\theta)}\right)\end{eqnarray*}
We now differentiate the $\frac{\sin(\alpha+\theta)}{\cos(2\alpha)-\cos(2\alpha+2\theta)}$
term above, which gives\begin{eqnarray*}
 &  & \frac{d}{d\theta}\left(\frac{\sin(\alpha+\theta)}{\cos(2\alpha)-\cos(2\alpha+2\theta)}\right)\\
 & = & \frac{1}{[\cos(2\alpha)-\cos(2\alpha+2\theta)]^{2}}\left[\cos(\alpha+\theta)[\cos(2\alpha)-\cos(2\alpha+2\theta)]-2\sin(2\alpha+2\theta)\sin(\alpha+\theta)\right]\end{eqnarray*}
The numerator is simplified to be:\begin{eqnarray*}
 &  & \cos(\alpha+\theta)[\cos(2\alpha)-\cos(2\alpha+2\theta)]-2\sin(2\alpha+2\theta)\sin(\alpha+\theta)\\
 & = & \cos(\alpha+\theta)[\cos(2\alpha)-2\cos^{2}(\alpha+\theta)+1]-4\sin^{2}(\alpha+\theta)\cos(\alpha+\theta)\\
 & = & \cos(\alpha+\theta)[\cos(2\alpha)-2\cos^{2}(\alpha+\theta)+1-4\sin^{2}(\alpha+\theta)]\\
 & = & \cos(\alpha+\theta)[\cos(2\alpha)-2\cos^{2}(\alpha+\theta)+4\cos^{2}(\alpha+\theta)-3]\\
 & = & \cos(\alpha+\theta)[\cos(2\alpha)+2\cos^{2}(\alpha+\theta)-3].\end{eqnarray*}
With this formula, we see that the conditions for $\frac{d}{d\theta}\left(\frac{\sin(\alpha+\theta)}{\cos(2\alpha)-\cos(2\alpha+2\theta)}\right)=0$
is to have $\cos(\alpha+\theta)=0$ or $\cos(2\alpha)+2\cos^{2}(\alpha+\theta)=3$.
The first case gives us $\theta=\frac{\pi}{2}-\alpha$, which gives
us the required conclusion. The second case requires $\alpha=0$ or
$\alpha=\pi$ and $\theta=0$ or $\theta=\pi$, which are degenerate
cases. This gives us all optimum solutions to our problem, and concludes
the proof.
\end{proof}
We now create an example in $\mathbb{R}^{3}$ that illustrates that
the omission of the condition of unique solutions need not give us
points whose gradients point in opposite directions.
\begin{example}
(Gradients need not be opposite) \label{exa:4-lines}Define the four
lines $L_{1}$ to $L_{4}$ by \begin{eqnarray*}
L_{1} & := & \{(0,0,-1)+\lambda(1,0,1)\mid\lambda\in\mathbb{R}_{+}\}\\
L_{2} & := & \{(0,0,-1)+\lambda(-1,0,1)\mid\lambda\in\mathbb{R}_{+}\}\\
L_{3} & := & \{(0,0,1)+\lambda(0,1,-1)\mid\lambda\in\mathbb{R}_{+}\}\\
L_{4} & := & \{(0,0,1)+\lambda(0,-1,-1)\mid\lambda\in\mathbb{R}_{+}\}\end{eqnarray*}
The lines $L_{1}$ and $L_{2}$ lie in the $x$-$z$ plane, while
the lines $L_{3}$ and $L_{4}$ lie in the $y$-$z$ plane. See the
diagram on the right of Figure \ref{fig:The-lines}.

Consider first the problem of finding a plane $S$ that is a minimizer
of the maximum of the distances between the points defined by the
intersections of $S$ and the $L_{i}$'s. We now show that $S$ has
to be the $x$-$y$ plane. The plane $S$ intersects the $z$ axis
at some point $(0,0,p)$. When $S$ is the $x$-$y$ plane, the maximum
distance between the points is $2$. By Lemma \ref{lem:intersecting-rays},
the distance between the points $S\cap L_{1}$ and $S\cap L_{2}$
is at least $2(1-p)$, while the distance between the the points $S\cap L_{3}$
and $S\cap L_{4}$ is at least $2(1+p)$. This tells us that the $x$-$y$
plane is optimal.

With this observation, we now construct our example. Consider the
function $f:\mathbb{R}^{3}\rightarrow\mathbb{R}$ defined by\[
f(x,y,z)=-\left(\frac{x}{1+z}+\frac{y}{1-z}\right)^{4/3}-\left(\frac{x}{1+z}-\frac{y}{1-z}\right)^{4/3}.\]
The level set $\lev_{\geq-2}f$ contains the lines $L_{1}$ to $L_{4}$.
This means that $\mbox{diam}(S\cap\lev_{\geq-2}f)\geq2$. This is
in fact an equation when $S$ is the $x$-$y$ plane, and the maximizers
being the pairs $\{\pm(1,0,0)\}$ and $\{\pm(0,1,0)\}$.

The gradient $\nabla f(x,y,z)$ is\[
\nabla f(x,y,z)=\left(\begin{array}{c}
-\frac{4}{3}\left(\frac{x}{1+z}+\frac{y}{1-z}\right)^{1/3}-\frac{4}{3}\left(\frac{x}{1+z}-\frac{y}{1-z}\right)^{1/3}\\
-\frac{4}{3}\left(\frac{x}{1+z}+\frac{y}{1-z}\right)^{1/3}+\frac{4}{3}\left(\frac{x}{1+z}-\frac{y}{1-z}\right)^{1/3}\\
-\frac{4}{3}\left(\frac{x}{1+z}+\frac{y}{1-z}\right)^{1/3}\left(-\frac{x}{(1+z)^{2}}+\frac{y}{(1-z)^{2}}\right)-\frac{4}{3}\left(\frac{x}{1+z}-\frac{y}{1-z}\right)^{1/3}\left(-\frac{x}{(1+z)^{2}}-\frac{y}{(1-z)^{2}}\right)\end{array}\right)\]
With this, we can evaluate $\nabla f$ at $\pm(1,0,0)$ and $\pm(0,1,0)$
to be \begin{eqnarray*}
\nabla f(1,0,0) & = & \left(-\frac{8}{3},0,\frac{8}{3}\right),\\
\nabla f(-1,0,0) & = & \left(\frac{8}{3},0,\frac{8}{3}\right),\\
\nabla f(0,1,0) & = & \left(0,-\frac{8}{3},-\frac{8}{3}\right),\\
\nabla f(0,-1,0) & = & \left(0,\frac{8}{3},-\frac{8}{3}\right).\end{eqnarray*}
Neither of the pairs $\{\pm(1,0,0)\}$ and $\{\pm(0,1,0)\}$ have
opposite pointing gradients, which concludes our example.
\end{example}

\section{\label{sec:Other-local-analysis}Another convergence property of
critical points}

In this section, we look at a condition on critical points similar
to Proposition \ref{pro:triples-conv} that can be helpful for numerical
methods for finding critical points. Theorem \ref{thm:crit-from-linking}
below does not seem to be easily found in the literature, and can
be seen as a local version of the mountain pass theorem. 

We prove Theorem \ref{thm:crit-from-linking} in the more general
setting of metric spaces. Such a treatment includes the case of nonsmooth
functions. We recall the following definitions in metric critical
point theory from \cite{DM94,IS96,Katriel94}.
\begin{defn}
\label{def:Deformation-critical}Let $(X,d)$ be a metric space. We
call the point $x$ \emph{Morse regular} for the function $f:X\rightarrow\mathbb{R}$
if, for some numbers $\gamma,\sigma>0$, there is a continuous function
\[
\phi:\mathbb{B}(x,\gamma)\times[0,\gamma]\rightarrow X\]
such that all points $u\in\mathbb{B}(x,\gamma)$ and $t\in[0,\gamma]$
satisfy the inequality \[
f\big(\phi(x,t)\big)\leq f(x)-\sigma t,\]
and that $\phi(\cdot,0):\mathbb{B}(x,\gamma)\to\mathbb{B}(x,\gamma)$
is the identity map. The point $x$ is \emph{Morse critical }if it
is not Morse regular. 

If for some $\phi$, there is some $\kappa>0$ such that $\phi$ also
satisfies the inequality \[
d\big(\phi(x,t),x\big)\leq\kappa t,\]
then we call $x$ \emph{deformationally regular}. The point $x$ is
\emph{deformationally critical }if it is not deformationally regular.
\end{defn}
It is a fact that if $X$ is a Banach space and $f$ is locally Lipschitz,
then deformationally critical points are Clarke critical. The following
theorem gives a strategy for identifying deformationally critical
points.
\begin{thm}
\label{thm:crit-from-linking}(Critical points from sequences of linking
sets) Let $X$ be a metric space and $f:X\to\mathbb{R}$. Suppose
there is some open set $U$ of $\bar{x}$ and sequences of sets $\{\Phi_{i}\}_{i=1}^{\infty}$
and $\{\Gamma_{i}\}_{i=1}^{\infty}$ such that 
\begin{enumerate}
\item $\Phi_{i}$ and $\partial\Gamma_{i}$ link.
\item $\Gamma_{i}$ are homeomorphic to $\mathbb{B}^{m}$ for all $i$,
and $\max_{x\in\partial\Gamma_{i}}f(x)<\inf_{x\in\Phi_{i}\cap U}f(x)$. 
\item For any open set $V$ containing $\bar{x}$, there is some $I>0$
such that $\Gamma_{i}\subset V$ for all $i>I$. 
\item $f$ is Lipschitz in $U$.
\end{enumerate}
Then $\bar{x}$ is deformationally critical.\end{thm}
\begin{proof}
Suppose $\bar{x}$ is deformationally regular. Then there are $\gamma,\sigma,\kappa>0$
and $\phi:\mathbb{B}(\bar{x},\gamma)\times[0,\gamma]\to X$ such that
the inequalities\[
f\big(\phi(x,t)\big)\leq f(x)-\sigma t\mbox{ and }d\big(\phi(x,t),x\big)\leq\kappa t\]
hold for all points $x\in\mathbb{B}(\bar{x},\gamma)$ and $t\in[0,\gamma]$,
and $\phi(\cdot,0):\mathbb{B}(\bar{x},\gamma)\to\mathbb{B}(\bar{x},\gamma)$
is the identity map. We may reduce $\gamma$ as necessary and assume
that $U=\mathbb{B}(\bar{x},\gamma)$. 

Condition (3) implies that for any $\alpha>0$, then there is some
$I_{1}$ such that $\Gamma_{i}\subset\mathbb{B}(\bar{x},\alpha)$
for all $i>I_{1}$. Consider $\Gamma_{i,t}:=\phi((\Gamma_{i}\times\{t\})\cup(\partial\Gamma_{i}\times[0,t]))$.
Provided $0<t<\frac{\gamma-\alpha}{\kappa}$, we have $\Gamma_{i,t}\subset U$.

Since $f$ is Lipschitz in $U$, let $\bar{\kappa}$ be the modulus
of Lipschitz continuity in $U$. We have $\max_{x\in\Gamma_{i}}f(x)\leq\max_{x\in\partial\Gamma_{i}}f(x)+\bar{\kappa}\diam(\Gamma_{i})$.
Also,\begin{eqnarray*}
\max_{x\in\Gamma_{i.t}}f(x) & \leq & \max\left(\max_{x\in\partial\Gamma_{i}}f(x)+\bar{\kappa}\diam(\Gamma_{i})-\sigma t,\max_{x\in\partial\Gamma_{i}}f(x)\right).\end{eqnarray*}

By condition (3), there is some $I_{2}>0$ such that if $i>I_{2}$,
then $\diam(\Gamma_{i})<\frac{\sigma t}{\bar{\kappa}}$. So for $i>\max(I_{1},I_{2})$,
we have \[
\max_{x\in\Gamma_{i,t}}f(x)=\max_{x\in\partial\Gamma_{i}}f(x)<\inf_{x\in\Phi_{i}\cap U}f(x).\]

However, the fact that $\partial\Gamma_{i}$ and $\Phi_{i}$ link
implies that $\Gamma_{i,t}$ and $\Phi_{i}$ must intersect, and since
$\Gamma_{i,t}\subset U$, $\Gamma_{i,t}$ and $\Phi_{i}\cap U$ must
intersect. This is a contradiction, so $\bar{x}$ is deformationally
critical.
\end{proof}
It is reasonable to choose $\Gamma_{i}$ to be a simplex (that is,
a convex hull of $m+1$ points) and $\Phi_{i}$ to be an affine space.
If the sequence of sets $\{\Gamma_{i}\}_{i=1}^{\infty}$ converges
to the single point $\bar{x}$ and $f$ is $\mathcal{C}^{2}$ there,
a quadratic approximation of $f$ using only the knowledge of the
values of $f$ and $\nabla f$ on the vertices of the simplex would
be good approximation of $f$ on the simplex. We outline our strategy
below.
\begin{algorithm}
\label{alg:quad-approx}(Obtaining unknowns in quadratic) Let $h:\mathbb{R}^{m}\to\mathbb{R}$
be defined by $h(x)=\frac{1}{2}x^{T}Ax+b^{T}x+c$, and let $p_{1},\dots,p_{m+1}$
be $m+1$ points in $\mathbb{R}^{m}$. Suppose that the values of
$h(p_{i})$ and $\nabla h(p_{i})$ are known for all $i=1,\dots,m+1$.
We seek to obtain the values of $A$, $b$ and $c$.
\begin{enumerate}
\item Let $P\in\mathbb{R}^{m\times m}$ be the matrix such that the $i$th
column is $p_{i+1}-p_{1}$, and let $D\in\mathbb{R}^{m\times m}$
be the matrix such that the $i$th column is $\nabla h(p_{i+1})-\nabla h(p_{1})$.
Calculate $A$ with $A=DP^{-1}$.
\item Calculate $b$ with $b=\nabla h(p_{1})-Ap_{1}$.
\item Calculate $c$ with $c=h(p_{1})-\frac{1}{2}p_{1}^{T}Ap_{1}-b^{T}p_{1}$.
\end{enumerate}
\end{algorithm}
If $h$ is $\mathcal{C}^{2}$ instead of being a quadratic, then the
procedure in Algorithm \ref{alg:quad-approx} can be used to approximate
the values of $h$ on a simplex. 

In Lemma \ref{lem:quad-est} below, given $m+1$ points in $\mathbb{R}^{n}$,
we need to approximate a quadratic function on $\Delta$ as a subset
of $\mathbb{R}^{n}$. Even though $m<n$, the procedure to obtain
a quadratic estimate of $f$ in $\Delta$ is a straightforward extension
of Algorithm \ref{alg:quad-approx}. 
\begin{lem}
\label{lem:quad-est}(Quadratic estimate on simplex) Let $f:\mathbb{R}^{n}\to\mathbb{R}$
be $\mathcal{C}^{2}$ and $\bar{x}\in\mathbb{R}^{n}$. Let $p_{1},\dots,p_{m+1}$
be points close to $\bar{x}$. Suppose the matrix $P\in\mathbb{R}^{n\times m}$,
whose $i$th column is $p_{i+1}-p_{1}$, has full column rank. Let
$f_{e}:\mathbb{R}^{n}\cap\Delta\to\mathbb{R}$ be defined as the quadratic
function obtained using $f(p_{i})$ and $\nabla f(p_{i})$ for $i=1,\dots,m+1$
with Algorithm \ref{alg:quad-approx}. For any $\epsilon>0$, there
is some $\delta>0$ such that if $p_{1},\dots,p_{m+1}\in\mathbb{B}(\bar{x},\delta)$,
then \[
|f_{e}(x)-f(x)|<\frac{1}{2}\diam(\Delta)^{2}\epsilon(1+\kappa\|P\|\|P^{\dagger}\|)\mbox{ for all }x\in\Delta,\]
where $\|\cdot\|$ stands for the matrix 2-norm, $P^{\dagger}$ is
the pseudoinverse of $P$, and $\kappa$ is some constant dependent
only on $n$ and $m$.\end{lem}
\begin{proof}
The first step of this proof is to show that step 1 of Algorithm \ref{alg:quad-approx}
gives a matrix in $\mathbb{R}^{n\times n}$ which is a good approximation
of how $A=\nabla^{2}f(\bar{x})$ acts on the lineality space of the
affine hull of $\Delta$. Since $f$ is $\mathcal{C}^{2}$, for any
$\epsilon>0$, there exists $\delta>0$ such that $|\nabla f(x)-\nabla f(x^{\prime})-A(x-x^{\prime})|<\epsilon|x-x^{\prime}|$
for all $x,x^{\prime}\in\mathbb{B}(\bar{x},\delta)$. Thus, there
is some $\kappa>0$ depending only on $m$ and $n$ such that if $p_{1},\dots,p_{m+1}\in\mathbb{B}(\bar{x},\delta)$,
then \begin{equation}
\|D-AP\|<\kappa\epsilon\|P\|.\label{eq:|D-AP|}\end{equation}

Let $P=QR$, where $Q\in\mathbb{R}^{n\times m}$ has orthonormal columns
and $R\in\mathbb{R}^{m\times m}$, be a QR decomposition of $P$.
For any $v\in\mathbb{R}^{n}$ in the range of $P$, or equivalently,
$v=Qv^{\prime}$ for some $v^{\prime}\in\mathbb{R}^{m}$, we want
to show that $\|Av-DR^{-1}Q^{T}v\|$ is small. We note that $|v|=|v^{\prime}|$,
and we have the following calculation. \begin{eqnarray*}
\|Av-DR^{-1}Q^{T}v\| & = & \|AQv^{\prime}-DR^{-1}Q^{T}Qv^{\prime}\|\\
 & = & \|AQv^{\prime}-DR^{-1}v^{\prime}\|\\
 & \leq & \|AQ-DR^{-1}\||v^{\prime}|\\
 & \leq & \|AQR-D\|\|R^{-1}\||v|\\
 & = & \|D-AP\|\|R^{-1}\||v|\\
 & \leq & \kappa\epsilon\|P\|\|R^{-1}\||v|.\end{eqnarray*}
Next, for $x,x^{\prime}\in\mathbb{B}(\bar{x},\delta)$, let $d=\unitt(x^{\prime}-x)$.
Then\begin{eqnarray*}
f(x^{\prime})-f(x) & = & \int_{0}^{|x^{\prime}-x|}\nabla f(x+sd)^{T}d\,\mathbf{d}s\\
 & = & \int_{0}^{|x^{\prime}-x|}\int_{0}^{s}d^{T}\nabla^{2}f(x+td)d\,\mathbf{d}t+\nabla f(x)^{T}d\,\mathbf{d}s.\end{eqnarray*}
Since $f$ is $\mathcal{C}^{2}$, we may reduce $\delta$ if necessary
so that $\|A-\nabla^{2}f(x+td)\|<\epsilon$ for all $0\leq t\leq|x^{\prime}-x|$.
This tells us that \begin{eqnarray*}
|d^{T}(DR^{-1}Q^{T})d-d^{T}\nabla^{2}f(x)d| & \leq & |d|\|DR^{-1}Q^{T}d-\nabla^{2}f(x)d\|\\
 & \leq & \|DR^{-1}Q^{T}d-Ad\|+\|Ad-\nabla^{2}f(x)d\|\\
 & \leq & \kappa\epsilon\|P\|\|R^{-1}\||d|+\|A-\nabla^{2}f(x)\|\\
 & \leq & \epsilon(1+\kappa\|P\|\|R^{-1}\|).\end{eqnarray*}
We have \begin{eqnarray*}
 &  & \int_{0}^{|x^{\prime}-x|}\int_{0}^{s}d^{T}(DR^{-1}Q^{T})d\,\mathbf{d}t+\nabla f(x)^{T}d\,\mathbf{d}s\\
 & = & \nabla f(x)^{T}(x^{\prime}-x)+\int_{0}^{|x^{\prime}-x|}d^{T}(DR^{-1}Q^{T})ds\,\mathbf{d}s\\
 & = & \nabla f(x)^{T}(x^{\prime}-x)+\frac{|x^{\prime}-x|^{2}}{2}d^{T}(DR^{-1}Q^{T})d\\
 & = & \nabla f(x)^{T}(x^{\prime}-x)+\frac{1}{2}(x^{\prime}-x)^{T}(DR^{-1}Q^{T})(x^{\prime}-x).\end{eqnarray*}
Continuing with the arithmetic earlier, we obtain\begin{eqnarray*}
 &  & \left|f(x^{\prime})-f(x)-\left(\nabla f(x)^{T}(x^{\prime}-x)+\frac{1}{2}(x^{\prime}-x)^{T}(DR^{-1}Q^{T})(x^{\prime}-x)\right)\right|\\
 & \leq & \int_{0}^{|x^{\prime}-x|}\int_{0}^{s}\epsilon(1+\kappa\|P\|\|R^{-1}\|)s\,\mathbf{d}t\mathbf{d}s\\
 & = & \frac{1}{2}|x^{\prime}-x|^{2}\epsilon(1+\kappa\|P\|\|R^{-1}\|).\end{eqnarray*}
Let $x=p_{1}$ and $x^{\prime}$ be any point in $\Delta$. Define
$f_{e}(x^{\prime})$ by \[
f_{e}(x^{\prime})=f(x)+\nabla f(x)^{T}(x^{\prime}-x)+\frac{1}{2}(x^{\prime}-x)^{T}(DR^{-1}Q^{T})(x^{\prime}-x),\]
which is the quadratic function obtained using Algorithm \ref{alg:quad-approx}.
We have \begin{eqnarray*}
|f_{e}(x^{\prime})-f(x^{\prime})| & \leq & \frac{1}{2}|x^{\prime}-x|^{2}\epsilon(1+\kappa\|P\|\|R^{-1}\|)\\
 & \leq & \frac{1}{2}\diam(\Delta)^{2}\epsilon(1+\kappa\|P\|\|R^{-1}\|),\end{eqnarray*}
which gives what we need.
\end{proof}
In the statement of Lemma \ref{lem:quad-est}, we chose the domain
of $f$ to be $\mathbb{R}^{n}$ so that the inequality \eqref{eq:|D-AP|}
follows from the equivalence of finite dimensional norms. Next, the
accuracy of the computed values of $\nabla f(p_{i+1})-\nabla f(p_{1})$
might be poor, which makes the quadratic approximation strategy ineffective
once we are too close to the critical point $\bar{x}$. We remark
on how we can overcome this problem by exploiting concavity.
\begin{rem}
(Exploiting concavity) The lineality space of the affine hull of $\Delta$
may span the eigenspaces of the $m$ negative eigenvalues of $\nabla^{2}f(\bar{x})$
once we are close to the critical point $\bar{x}$. This can be checked
by calculating the Hessian as was done earlier. If this is the case,
$f$ would be concave in $\Delta$ when $p_{1},\dots,p_{m+1}$ are
sufficiently close to $\bar{x}$. The estimate $f(x)\leq f(p_{i})+\nabla f(p_{i})^{T}(x-p_{i})$
would hold for all $x\in\Delta$ and $1\leq i\leq m+1$, which can
give a sufficiently good estimate of $\max_{x\in\partial\Delta}f(x)$
through linear programming.
\end{rem}

\section{Fast local convergence\label{sec:Fast-local-convergence}}

In this section, we discuss how we can find good lower bounds that
allow us to achieve better convergence if $f$ is $\mathcal{C}^{2}$
and $X=\mathbb{R}^{n}$. Our method extends the local superlinearly
convergent method in \cite{LP08} for finding smooth critical points
of mountain pass type when $X=\mathbb{R}^{n}$.

Let us recall the multidimensional mountain pass theorem due to Rabinowitz
\cite{R77}.
\begin{thm}
(Multidimensional mountain pass theorem) \cite{R77}\label{thm:Rabinowitz-77}
Let $X=Y\oplus Z$ be a Banach space with $Z$ closed in $X$ and
$\dim(Y)<\infty$. For $\rho>0$ define\[
\mathcal{M}:=\{u\in Y\mid\|u\|\leq\rho\},\quad\mathcal{M}_{0}:=\{u\in Y\mid\|u\|=\rho\}.\]
Let $f:X\to\mathbb{R}$ be $\mathcal{C}^{1}$, and \[
b:=\inf_{u\in Z}f(u)>a:=\max_{u\in\mathcal{M}_{0}}f(u).\]
If $f$ satisfies the Palais Smale condition and \[
c:=\inf_{\gamma\in\Gamma}\max_{u\in\mathcal{M}}f\big(\gamma(u)\big)\mbox{ where }\Gamma:=\left\{ \gamma:\mathcal{M}\to X\mbox{ is continuous}\mid\gamma|_{\mathcal{M}_{0}}=id\right\} ,\]
then $c$ is a critical value of $f$. 
\end{thm}
 For the case when $X=\mathbb{R}^{3}$, we have an illustration
in Figure \ref{fig:multi-MPT-fig} of the case $f:\mathbb{R}^{3}\rightarrow\mathbb{R}$
defined by $f(x)=x_{3}^{2}-x_{1}^{2}-x_{2}^{2}$. The critical point
$\mathbf{0}$ has critical value $0$. Choose $Y$ to be $\mathbb{R}^{2}\times\{0\}$
and $Z$ to be $\{\mathbf{0}\}\times\mathbb{R}$. The union of the
two blue cones is the level set $\mbox{lev}_{=0}f:=f^{-1}(0)$, while
the bold red ring denotes $\mathcal{M}_{0}$ and the red disc denotes
a possible image of $\mathcal{M}$ under $\gamma$. 

\begin{figure}

\includegraphics[scale=0.5]{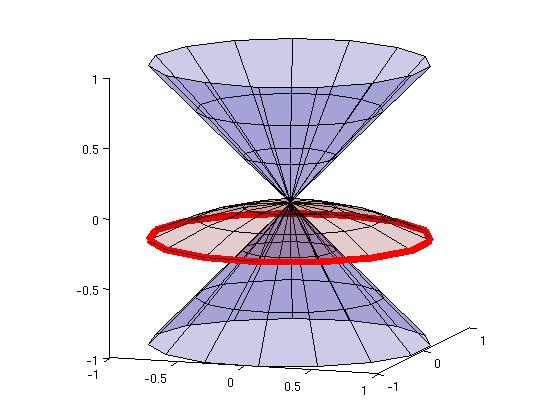}\caption{\label{fig:multi-MPT-fig}Illustration of the multidimensional mountain
pass theorem}

\end{figure}

It seems intuitively clear that $\gamma(\mathcal{M})$ has to intersect
the vertical axis. This is indeed the case, since $\mathcal{M}_{0}$
and $Z$ link. (See for example \cite[Example II.8.2]{Str08}.)

With this observation, we easily see that $\max_{u\in\mathcal{M}}f(\gamma(u))\geq\inf_{z\in Z}f(z)$.
Thus the critical value $c=\inf_{\gamma\in\Gamma}\max_{u\in\mathcal{M}}f(\gamma(u))$
from Theorem \ref{thm:Rabinowitz-77} is bounded from below by $\inf_{z\in Z}f(z)$.
This gives a lower bound for the critical value. In the mountain pass
case when $m=1$, the set $\mathcal{M}_{0}$ consists of two points,
and the space $Z$ separates the two points in $\mathcal{M}_{0}$
so that any path connecting the two points in $\mathcal{M}_{0}$ must
intersect $Z$. 

A first try for a fast locally convergent algorithm is as follows:
\begin{algorithm}
\label{alg:first-local-algorithm}A first try for a fast locally convergent
algorithm to find saddle points of Morse index $m$ for $f:X\to\mathbb{R}$.\end{algorithm}
\begin{enumerate}
\item Set the iteration count $i$ to $0$, and let $l_{i}$ be a lower
bound of the critical value. 
\item Find $x_{i}$ and $y_{i}$, where $(S_{i},x_{i},y_{i})$ is an optimizing
triple of\begin{equation}
\min_{S\in\mathcal{S}}\max_{x,y\in S\cap(\scriptsize\lev_{\geq l_{i}}f)\cap U_{i}}|x-y|,\label{eq:key-exp2}\end{equation}
where $U_{i}$ is an open set. Here $\mathcal{S}$ is the set of $m$-dimensional
affine subspaces of $\mathbb{R}^{n}$ intersecting $U_{i}$. (The
difference between this formula and \eqref{eq:keyexp1} is that we
take level sets of level $l_{i}$ instead of $\frac{1}{2}(l_{i}+u_{i})$.)

\item For an optimal solution $(S_{i},x_{i},y_{i})$, let $l_{i+1}$ be
the lower bound of $f$ on the $(n-m)$-dimensional affine space passing
through $z_{i}:=\frac{1}{2}(x_{i}+y_{i})$ whose lineality space is
orthogonal to the lineality space of $S_{i}$ .
\item Increase $i$ and go back to step 2.
\end{enumerate}
While Algorithm \ref{alg:first-local-algorithm} as stated works fine
for the case $m=1$ to find critical points of mountain pass type,
the $l_{i}$'s calculated in this manner need not increase monotonically
to the critical value when $m>1$. We first present a lemma on the
min-max problem \eqref{eq:key-exp2} for the case of a quadratic.
\begin{lem}
(Analysis on exact quadratic) \label{lem:quadratic-min-max} Consider
$f:\mathbb{R}^{n}\rightarrow\mathbb{R}$ defined by $f(x)=\sum_{j=1}^{n}a_{j}x_{j}^{2}$,
where $a_{i}$ are in decreasing order, with $a_{j}>0$ for $1\leq j\leq n-m$
and $a_{j}<0$ for $n-m+1\leq j\leq n$. The function $f$ has one
critical point $\mathbf{0}$, and $f(\mathbf{0})=0$. Given $l<0$,
an optimizing triple $(\bar{S},\bar{x},\bar{y})$ of the problem \begin{equation}
\min_{S\in\mathcal{S}}\max_{x,y\in S\cap\scriptsize\lev_{\geq l}f}|x-y|,\label{eq:quadratic-min-max}\end{equation}
where $\mathcal{S}$ is the set of affine spaces of dimension $m$,
satisfies \[
\bar{x}=\left(0,0,\dots,0,\pm\sqrt{\frac{l}{a_{n-m+1}}},0,\dots,0\right),\]
where the nonzero term is in the $(n-m+1)$th position, and $\bar{y}=-\bar{x}$. \end{lem}
\begin{proof}
Let $S_{\bar{z},V}:=\{\bar{z}+Vw\mid w\in\mathbb{R}^{m}\}$, where
$V\in\mathbb{R}^{n\times m}$ is a matrix with orthonormal columns.
Let the matrix $A\in\mathbb{R}^{n\times n}$ be the diagonal matrix
with entries $a_{j}$ in the $(j,j)$th position. The ellipse $S_{\bar{z},V}\cap\lev_{\geq l}f$
can be written as a union of elements of the form $\bar{z}+Vw$, where
$w$ satisfies \begin{eqnarray*}
(\bar{z}+Vw)^{T}A(\bar{z}+Vw) & \geq & l\\
\Leftrightarrow w^{T}V^{T}AVw+2\bar{z}^{T}AVw+\bar{z}^{T}A\bar{z} & \geq & l.\end{eqnarray*}
If the matrix $V^{T}AV$ has a nonnegative eigenvalue, then $S_{\bar{z},V}\cap\lev_{\geq l}f$
is unbounded. Otherwise, the set \[
\{\bar{z}+Vw\mid w^{T}V^{T}AVw+2\bar{z}^{T}AVw+\bar{z}^{T}A\bar{z}\geq l\}\]
is bounded. Therefore the inner maximization problem of \eqref{eq:quadratic-min-max}
corresponding to $S=S_{\bar{z},V}$ has a (not necessarily unique)
pair of minimizers. We continue completing the square with respect
to $w$ and let the symmetric matrix $C$ be the square root $C=[-V^{T}AV]^{\frac{1}{2}}$.\begin{eqnarray*}
-w^{T}C^{2}w+2\bar{z}^{T}AVw+\bar{z}^{T}A\bar{z} & \geq & l\\
\Leftrightarrow-(Cw-C^{-1}V^{T}A\bar{z})^{T}(Cw-C^{-1}V^{T}A\bar{z})+\bar{z}^{T}A\bar{z}+\bar{z}^{T}AVC^{-2}V^{T}A^{T}\bar{z} & \geq & l.\end{eqnarray*}
The maximum length between two points of an ellipse is twice the distance
between the center and the furthest point on the ellipse. (This fact
is easily proved by reducing to, and examining, the two dimensional
case.) The distance between the center and the furthest point on
the ellipse $S_{\bar{z},V}\cap\lev_{\geq l}f$ can be calculated to
be\[
\sqrt{\frac{1}{\alpha}(\bar{z}^{T}A\bar{z}+\bar{z}^{T}AVC^{-2}V^{T}A^{T}\bar{z}-l)},\]
where $\alpha$ is the square of the smallest eigenvalue in $C$,
or equivalently the negative of the largest eigenvalue of $V^{T}AV$.
The term $(\bar{z}^{T}A\bar{z}+\bar{z}^{T}AVC^{-2}V^{T}A^{T}\bar{z})$
is $\max\{f(x)\mid x\in S_{\bar{z},V}\}$, which we refer to as $\max_{S_{\bar{z},V}}f$.
We now proceed to minimize $\max_{S_{\bar{z},V}}f$ and maximize $\alpha$
separately.

\textbf{Claim 1: $\max_{S_{\bar{z},V}}f\geq0$.}

We first prove that the subspace \[
Z:=\{z\mid z_{n}=z_{n-1}=\cdots=z_{n-m+1}=0\}\]
must intersect $S_{\bar{z},V}$. Recall that $V^{T}AV$ is negative
definite. Therefore for any $w\neq\mathbf{0}$, $w^{T}V^{T}AVw<0$.
Since the first $n-m$ eigenvalues of $A$ are positive, $Vw$ cannot
be all zeros in its last $m$ components. This shows that the $m\times m$
matrix $V((n-m+1):n,1:m)$ is invertible. We can find some $\bar{w}$
such that the last $m$ components of $\bar{z}+V\bar{w}$ are zeros.
This shows that $S_{\bar{z},V}\cap Z\neq\emptyset$, so $\max_{S_{\bar{z},V}}f\geq\min\{f(x)\mid x\in Z\}=0$.

\textbf{Claim 2: $\alpha\leq-a_{n-m+1}$.}

To find the maximum value of $\alpha$, we recall that it is the negative
of the largest eigenvalue of $V^{T}AV$. Since $V\in\mathbb{R}^{n\times m}$,
the Courant-Fischer Theorem, gives $\alpha\leq-a_{n-m+1}$.

Choose the affine space $\bar{S}:=\{\mathbf{0}\}\times\mathbb{R}^{m}$.
This minimizes $\max_{S}f$ and maximizes $\alpha$ as well, giving
the optimal solution in the statement of the lemma.
\end{proof}
It should be noted however that the minimizing subspace need not be
unique, even if the values of $a_{j}$ are distinct. The example below
highlights how Algorithm \ref{alg:first-local-algorithm} can fail.
\begin{example}
(Failure of Algorithm \ref{alg:first-local-algorithm}) \label{exa:failure-first-try}Suppose
$f(x)=x_{1}^{2}-x_{2}^{2}-3x_{3}^{2}$. The subspace $S=\{x\mid x_{1}=x_{3}\}$
intersects the level set $\lev_{\geq-1}f$ in the disc \[
\left\{ \lambda\left(\frac{1}{\sqrt{2}}\sin\theta,\cos\theta,\frac{1}{\sqrt{2}}\sin\theta\right)\mid0\leq\theta\leq2\pi,0\leq\lambda\leq1\right\} .\]
The largest distance between two points on the disc is $2$, and the
subspace $S$ can be verified to give the optimal value to the min-max
problem \eqref{eq:key-exp2} by Lemma \ref{lem:quadratic-min-max}.

On the ray $S^{\perp}=\{\lambda(1,0,-1)\mid\lambda\in\mathbb{R}\}$,
the function $f$ is concave, hence there is no minimum. This example
illustrates that Algorithm \ref{alg:first-local-algorithm} can fail
in general. See Figure \ref{fig:bad-3d}.
\end{example}
\begin{figure}
\includegraphics[scale=0.4]{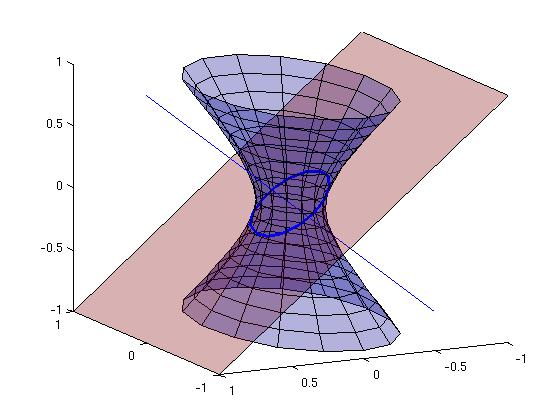}

\caption{\label{fig:bad-3d}An example where Algorithm \ref{alg:first-local-algorithm}
fails.}

\end{figure}

Example \ref{exa:failure-first-try} shows that even if there are
only 2 negative eigenvalues, it might be possible to find a two-dimensional
subspace $S$ on which the Hessian is negative definite on both $S$
and $S^{\perp}$. Therefore, we amend Algorithm \ref{alg:first-local-algorithm}
by determining the eigenspace corresponding to the $m$ smallest eigenvalues. 
\begin{algorithm}
\label{alg:fast-local-method} Fast local method to find saddle points
of Morse index $m$.\end{algorithm}
\begin{enumerate}
\item Set the iteration count $i$ to $0$, and let $l_{i}$ be a lower
bound of the critical value. 
\item Find $x_{i}$ and $y_{i}$, where $(S_{i}^{\prime},x_{i},y_{i})$
is an optimizing triple of\begin{equation}
\min_{S\in\mathcal{S}}\max_{x,y\in S\cap(\scriptsize\lev_{\geq l_{i}}f)\cap U_{i}}|x-y|,\label{eq:key-exp}\end{equation}
where $U_{i}$ is an open set. Here $\mathcal{S}$ is the set of $m$-dimensional
affine subspaces of $\mathbb{R}^{n}$ intersecting $U_{i}$. We emphasize
that the space where minimality is attained in the outer minimization
problem is $S_{i}^{\prime}$. After solving the above problem, find
the subspace $S_{i}$ that approximates the eigenspace corresponding
to the $m$ smallest eigenvalues using Algorithm \ref{alg:find-S-prime}
below.
\item For an optimizing triple $(S_{i},x_{i},y_{i})$ found in step 2, let
$l_{i+1}$ be the lower bound of $f$ on the $(n-m)$-dimensional
affine space passing through $z_{i}:=\frac{1}{2}(x_{i}+y_{i})$ whose
lineality space is orthogonal to the lineality space of $S_{i}$ .
\item Increase $i$ and go back to step 2 till convergence.
\end{enumerate}
A local algorithm is needed in Algorithm \ref{alg:fast-local-method}
to find the subspace $S_{i}$ in step 2.
\begin{algorithm}
\label{alg:find-S-prime}Finding the subspace $S_{i}$ in step 2 of
Algorithm \ref{alg:fast-local-method}:\end{algorithm}
\begin{enumerate}
\item Let $X_{1}=\mbox{span}\{x_{i}-y_{i}\}$, where $x_{i}$, $y_{i}$
are found in step 2 of Algorithm \ref{alg:fast-local-method}, and
let $j$ be $1$.
\item Find the closest point from $z_{i}:=\frac{1}{2}(x_{i}+y_{i})$ to
$\lev_{\leq l_{i}}f\cap(z_{i}+X_{j}^{\perp})$, which we call $\bar{p}_{j+1}$.
\item Let $X_{j+1}=\mbox{span}\{X_{j},\bar{p}_{j+1}-z_{i}\}$ and increase
$j$ by $1$. If $j=m$, let $S_{i}$ be $z_{i}+X_{m}$ and the algorithm
ends. Otherwise, go back to step 2.
\end{enumerate}
Step 2 of Algorithm \ref{alg:find-S-prime} finds the negative eigenvalues
and eigenvectors, starting from the eigenvalues furthest from zero.
Once all the eigenvectors are found, then $S_{i}$ is the span of
these eigenvectors. 

In some situations, the lineality space of $S_{i}$ are known in advance,
or do not differ too much from the previous iteration. In this case,
we can get around using Algorithm \ref{alg:find-S-prime} and use
the estimate instead.

We are now ready to prove the convergence of Algorithm \ref{alg:fast-local-method}.

\section{Proof of superlinear convergence of local algorithm\label{sec:Proof-of-superlinear}}

 We prove our result on the convergence of Algorithm \ref{alg:fast-local-method}
in steps. The first step is to look closely at a model problem. 
\begin{assumption}
\label{ass:condns-on-model}Given $\delta>0$, suppose $h:\mathbb{R}^{n}\rightarrow\mathbb{R}$
is $\mathcal{C}^{2}$, \begin{eqnarray*}
|\nabla h(x)-Ax| & \leq & \delta|x|,\\
\mbox{{\rm and }}|h(x)-\frac{1}{2}x^{T}Ax| & \leq & \frac{1}{2}\delta|x|^{2}\mbox{ {\rm for all }}x\in\mathbb{B},\end{eqnarray*}
where $A\in\mathbb{R}^{n\times n}$ is an invertible diagonal matrix
with diagonal entries ordered decreasingly, of which $a_{i}=A_{ii}$
and \[
a_{1}>a_{2}>\cdots>a_{n-m}>0>a_{n-m+1}>\cdots>a_{n}.\]
Define $h_{\min}:\mathbb{R}^{n}\rightarrow\mathbb{R}$ and $h_{\max}:\mathbb{R}^{n}\rightarrow\mathbb{R}$
by:\[
h_{\min}(x)=\frac{1}{2}x^{T}(A-\delta I)x\quad\mbox{and}\quad h_{\max}(x)=\frac{1}{2}x^{T}(A+\delta I)x.\]

\end{assumption}
It is clear that $\nabla h(\mathbf{0})=\mathbf{0}$, $h(\mathbf{0})=0$,
$\nabla^{2}h(\mathbf{0})=A$, and the Morse index is $m$. Here is
a simple observation that bounds the level sets of $h$:
\begin{prop}
(Level set property) The level sets of $h$ satisfy\begin{eqnarray*}
 &  & \mathbb{B}\cap\lev_{\geq l}h_{\min}\subset\mathbb{B}\cap\lev_{\geq l}h\subset\mathbb{B}\cap\lev_{\geq l}h_{\max},\\
 & \mbox{{\rm and}} & \mathbb{B}\cap\lev_{\leq l}h_{\max}\subset\mathbb{B}\cap\lev_{\leq l}h\subset\mathbb{B}\cap\lev_{\leq l}h_{\min}.\end{eqnarray*}
\end{prop}
\begin{proof}
This follows easily from $|h(x)-\frac{1}{2}x^{T}Ax|\leq\frac{1}{2}\delta|x|^{2}$
for all $x\in\mathbb{B}$.
\end{proof}
For convenience, we highlight the standard problem below:
\begin{problem}
\label{pro:standard-prob}Suppose $g:\mathbb{R}^{n}\rightarrow\mathbb{R}$
is $\mathcal{C}^{2}$, with critical point $\mathbf{0}$ of Morse
index $m$, $g(\mathbf{0})=0$ and the Hessian $\nabla^{2}g(\mathbf{0})$
has distinct eigenvalues that are all nonzero. Consider the problem
\[
\min_{S\in\mathcal{S}}\max_{x,y\in S\cap(\scriptsize\lev_{\geq l}g)\cap\mathbb{B}}|x-y|,\]
where $\mathcal{S}$ is the set of $m$ dimensional affine subspaces.
\end{problem}
Note that in Problem \ref{pro:standard-prob}, we have limited the
region where $x$ and $y$ lie in by $\mathbb{B}$. Here is a result
on the optimizing pair $(\bar{x},\bar{y})$ of the inner maximization
problem in Problem \ref{pro:standard-prob}.
\begin{lem}
(Convergence to eigenvector and saddle point)\label{lem:min-maximizer-ppties}For
all $\delta>0$ sufficiently small, suppose that $h:\mathbb{R}^{n}\rightarrow\mathbb{R}$
is such that Assumption \ref{ass:condns-on-model} holds. Assume that
for the optimizing triple $(\bar{S},\bar{x},\bar{y})$ of Problem
\ref{pro:standard-prob} for $g=h$, $(\bar{x},\bar{y})$ is the unique
pair of points in $\bar{S}\cap(\lev_{\geq l}h)\cap\mathbb{B}$ such
that $|\bar{x}-\bar{y}|=\diam(\bar{S}\cap(\lev_{\geq l}h)\cap\mathbb{B})$.
Then there exists $\epsilon>0$ such that if $l<0$ satisfies $-\epsilon<l<0$,
then $(\bar{x},\bar{y})$ are such that $\frac{\bar{y}-\bar{x}}{|\bar{y}-\bar{x}|}$
converges to the $(n-m+1)$th eigenvector as $\delta\to0$, and $|\frac{1}{2}(\bar{x}+\bar{y})|^{2}/|l|\to0$
as $\delta\to0$.\end{lem}
\begin{proof}
Since $\lev_{\leq l}h_{\max}\cap\mathbb{B}\subset\lev_{\leq l}h\cap\mathbb{B}$,
we look at the the optimal solution of the min-max problem for $h_{\max}$
first. The objective of Problem \ref{pro:standard-prob} for $g=h_{\max}$
is $2\sqrt{\frac{l}{a_{n-m+1}+\delta}}$ by Lemma \ref{lem:quadratic-min-max}.
This gives an upper bound for the min-max problem for $h$. Similarly,
by considering $h_{\min}$ instead, we deduce that Problem \ref{pro:standard-prob}
has optimal solution bounded from below by $2\sqrt{\frac{l}{a_{n-m+1}-\delta}}$.

Recall the optimality condition in Proposition \ref{pro:opposite-directions}.
We now proceed to find the first pair of points with\textbf{ }opposite
pointing gradients. Let $\bar{x}$ and $\bar{y}$ be optimal points
at level $l_{i}$. Now, \begin{eqnarray*}
\bar{y}-\bar{x} & = & \lambda_{1}\nabla h(\bar{x}),\\
\mbox{ and }\bar{x}-\bar{y} & = & \lambda_{2}\nabla h(\bar{y}),\end{eqnarray*}
 for some $\lambda_{1},\lambda_{2}>0$. Then \begin{eqnarray*}
\lambda_{1}\nabla h(\bar{x})+\lambda_{2}\nabla h(\bar{y}) & = & \mathbf{0}\\
|\lambda_{1}A\bar{x}+\lambda_{2}A\bar{y}| & \leq & \lambda_{1}|A\bar{x}-\nabla h(\bar{x})|+\lambda_{2}|A\bar{y}-\nabla h(\bar{y})|\\
 & \leq & \delta(\lambda_{1}|\bar{x}|+\lambda_{2}|\bar{y}|)\\
|A(\lambda_{1}\bar{x}+\lambda_{2}\bar{y})| & \leq & \delta(\lambda_{1}|\bar{x}|+\lambda_{2}|\bar{y}|)\\
\Rightarrow|\lambda_{1}\bar{x}+\lambda_{2}\bar{y}| & \leq & |A^{-1}||A(\lambda_{1}\bar{x}+\lambda_{2}\bar{y})|\\
 & \leq & |A^{-1}|\delta(\lambda_{1}|\bar{x}|+\lambda_{2}|\bar{y}|).\end{eqnarray*}
This means that there are points $x^{\prime}$ and $y^{\prime}$ such
that $\lambda_{1}x^{\prime}+\lambda_{2}y^{\prime}=\mathbf{0}$, $|\bar{x}-x^{\prime}|\leq|A^{-1}|\delta|\bar{x}|$
and $|\bar{y}-y^{\prime}|\leq|A^{-1}|\delta|\bar{y}|$. With this,
we now concentrate on pairs of points that are negative multiples
of each other.

Now,\begin{eqnarray*}
\lambda_{1}\nabla h(\bar{x}) & = & \bar{y}-\bar{x}\\
\Rightarrow\nabla h(\bar{x}) & = & \frac{1}{\lambda_{1}}(\bar{y}-\bar{x})\\
\Rightarrow\left|A\bar{x}-\frac{1}{\lambda_{1}}(\bar{y}-\bar{x})\right| & \leq & \delta|\bar{x}|.\end{eqnarray*}
Similarly, this gives us\begin{eqnarray*}
\left|A(-\bar{y})-\frac{1}{\lambda_{2}}(\bar{y}-\bar{x})\right| & \leq & \delta|\bar{y}|.\end{eqnarray*}
Therefore, \[
\left|A(\bar{y}-\bar{x})+\left(\frac{1}{\lambda_{1}}+\frac{1}{\lambda_{2}}\right)(\bar{y}-\bar{x})\right|\leq\delta(|\bar{x}|+|\bar{y}|).\]
This gives:\begin{eqnarray}
 &  & \left|A(y^{\prime}-x^{\prime})+\left(\frac{1}{\lambda_{1}}+\frac{1}{\lambda_{2}}\right)(y^{\prime}-x^{\prime})\right|\nonumber \\
 & \leq & |A(y^{\prime}-x^{\prime})-A(\bar{y}-\bar{x})|+\left(\frac{1}{\lambda_{1}}+\frac{1}{\lambda_{2}}\right)|(y^{\prime}-x^{\prime})-(\bar{y}-\bar{x})|\nonumber \\
 &  & \qquad+\left|A(\bar{y}-\bar{x})+\left(\frac{1}{\lambda_{1}}+\frac{1}{\lambda_{2}}\right)(\bar{y}-\bar{x})\right|\nonumber \\
 & \leq & |A||A^{-1}|\delta(|\bar{x}|+|\bar{y}|)+\left(\frac{1}{\lambda_{1}}+\frac{1}{\lambda_{2}}\right)|A^{-1}|\delta(|\bar{x}|+|\bar{y}|)+\delta(|\bar{x}|+|\bar{y}|)\nonumber \\
 & = & \delta(|\bar{x}|+|\bar{y}|)\underbrace{\left(|A||A^{-1}|+\left(\frac{1}{\lambda_{1}}+\frac{1}{\lambda_{2}}\right)|A^{-1}|+1\right)}_{(1)}.\label{eq:y_minus_x_eigenvalue}\end{eqnarray}
Next, we relate $|\bar{y}-\bar{x}|$ and $|\bar{x}|+|\bar{y}|$. We
have \begin{eqnarray*}
|\bar{x}| & \leq & |x^{\prime}|+|A^{-1}|\delta|\bar{x}|\\
\Rightarrow(1-|A^{-1}|\delta)|\bar{x}| & \leq & |x^{\prime}|,\end{eqnarray*}
and similarly, $(1-|A^{-1}|\delta)|\bar{y}|\leq|y^{\prime}|$. It
is clear that $|y^{\prime}-x^{\prime}|=|x^{\prime}|+|y^{\prime}|$,
and we get\begin{eqnarray*}
|\bar{x}|+|\bar{y}| & \leq & \frac{1}{1-|A^{-1}|\delta}(|x^{\prime}|+|y^{\prime}|)\\
 & = & \frac{1}{1-|A^{-1}|\delta}|y^{\prime}-x^{\prime}|\\
 & \leq & \frac{1}{1-|A^{-1}|\delta}[|\bar{y}-\bar{x}|+|A^{-1}|\delta(|\bar{x}|+|\bar{y}|)]\\
\Rightarrow\left(1-\frac{|A^{-1}|\delta}{1-|A^{-1}|\delta}\right)(|\bar{x}|+|\bar{y}|) & \leq & \frac{1}{1-|A^{-1}|\delta}|\bar{y}-\bar{x}|\\
\Rightarrow(1-2|A^{-1}|\delta)(|\bar{x}|+|\bar{y}|) & \leq & |\bar{y}-\bar{x}|.\end{eqnarray*}
To show that $(1)$ in \eqref{eq:y_minus_x_eigenvalue} converges
to $0$ as $\delta\searrow0$, we need to show that $(\frac{1}{\lambda_{1}}+\frac{1}{\lambda_{2}})$
remains bounded as $\delta\searrow0$. Note that \begin{eqnarray*}
\nabla h(\bar{x})-\nabla h(\bar{y}) & = & \left(\frac{1}{\lambda_{1}}+\frac{1}{\lambda_{2}}\right)(\bar{y}-\bar{x})\\
\Rightarrow\left|\frac{1}{\lambda_{1}}+\frac{1}{\lambda_{2}}\right| & = & \frac{\left|\nabla h(\bar{x})-\nabla h(\bar{y})\right|}{|\bar{y}-\bar{x}|}\\
 & \leq & \frac{1}{|\bar{y}-\bar{x}|}[|A(\bar{x}-\bar{y})|+\delta(|\bar{x}|+|\bar{y}|)]\\
 & \leq & \frac{1}{|\bar{y}-\bar{x}|}\left(|A||\bar{x}-\bar{y}|+\frac{1}{1-2|A^{-1}|\delta}\delta|\bar{y}-\bar{x}|\right)\\
 & = & |A|+\frac{\delta}{1-2|A^{-1}|\delta}.\end{eqnarray*}
Since the eigenvectors depend continuously on the entries of a matrix
when the eigenvalues remain distinct, we see that $\frac{1}{|\bar{y}-\bar{x}|}(y^{\prime}-x^{\prime})$
converges to an eigenvector of $A$ as $\delta\rightarrow0$ from
formula \eqref{eq:y_minus_x_eigenvalue}. 

Next, we show that $\frac{1}{|\bar{y}-\bar{x}|}(\bar{y}-\bar{x})$
converges to an eigenvector corresponding to the eigenvalue $a_{n-m+1}$.
Recall that $2\sqrt{\frac{l}{a_{n-m+1}-\delta}}\leq|\bar{x}-\bar{y}|$
and $\left|(\bar{x}-\bar{y})-(x^{\prime}-y^{\prime})\right|\leq2|A^{-1}|\delta$.
 So $\frac{1}{|\bar{y}-\bar{x}|}(\bar{y}-\bar{x})$ has the same
limit as $\frac{1}{|\bar{y}-\bar{x}|}(y^{\prime}-x^{\prime})$. If
$x^{\prime}$ and $y^{\prime}$ are such that $\frac{1}{|\bar{y}-\bar{x}|}(y^{\prime}-x^{\prime})$
converges to a eigenvector corresponding to $a_{k}$, then Lemma \ref{lem:bounds-on-primes}
below gives us the following chain of inequalities: \begin{eqnarray*}
|\bar{x}-\bar{y}| & \leq & |\bar{x}|+|\bar{y}|\\
 & \leq & 2\sqrt{[(1+\theta)^{2}+(n-1)\theta^{2}]\frac{l}{(a_{k}+\delta)(1-\theta)^{2}+(n-1)(a_{1}+\delta)\theta^{2}}},\end{eqnarray*}
where $\theta\to0$ as $\delta\searrow0$. We note that $k\geq n-m+1$
because $a_{k}$ cannot be nonnegative. As $\delta\searrow0$, the
limit of the RHS of the above is $2\sqrt{\frac{l}{a_{k}}}$. This
gives a contradiction if $k>n-m+1$, so $k=n-m+1$.

\textbf{To show $\frac{1}{2}(\bar{x}+\bar{y})\to\mathbf{0}$ as $\delta\searrow0$:
}We now work out an upper bound for $\left|\frac{1}{2}(\bar{x}+\bar{y})\right|$
using Lemma \ref{lem:bounds-on-primes}. We get\begin{eqnarray*}
\left|\frac{1}{2}(\bar{x}+\bar{y})\right| & \leq & \left|\frac{1}{2}(x^{\prime}+y^{\prime})\right|+\frac{1}{2}|A^{-1}|\delta(|\bar{x}|+|\bar{y}|)\\
 & = & \frac{1}{2}\big||x^{\prime}|-|y^{\prime}|\big|+\frac{1}{2}|A^{-1}|\delta(|\bar{x}|+|\bar{y}|)\\
 & \leq & \frac{1}{2}\big||\bar{x}|-|\bar{y}|\big|+|A^{-1}|\delta(|\bar{x}|+|\bar{y}|)\\
 & \leq & \frac{1}{2}\sqrt{[(1+\theta)^{2}+(n-1)\theta^{2}]\frac{l}{(a_{n-m+1}+\delta)(1-\theta)^{2}+(n-1)(a_{1}+\delta)\theta^{2}}}\\
 &  & \qquad-\frac{1}{2}[1-\theta]\sqrt{\frac{l}{(a_{n-m+1}-\delta)(1+\theta)^{2}+(n-1)(a_{n}-\delta)\theta^{2}}}+|A^{-1}|\delta(|\bar{x}|+|\bar{y}|).\end{eqnarray*}
Here, $\theta>0$ is such that $\theta\to0$ as $\delta\to0$. At
this point, we note that the final formula above can be written as
$\frac{1}{2}\left(\sqrt{\frac{l}{c_{1}}}-\sqrt{\frac{l}{c_{2}}}\right)+|A^{-1}|\delta(|\bar{x}|+|\bar{y}|)$,
where $c_{1},c_{2}<0$, with $|c_{1}|<|c_{2}|$, and $c_{1},c_{2}\rightarrow a_{n-m+1}$
as $\delta\rightarrow0$. Therefore\begin{eqnarray*}
\left|\frac{1}{2}(\bar{x}+\bar{y})\right| & \leq & \frac{1}{2}\left(\sqrt{\frac{l}{c_{1}}}-\sqrt{\frac{l}{c_{2}}}\right)+|A^{-1}|\delta(|\bar{x}|+|\bar{y}|)\\
 & = & \frac{1}{2}\frac{\frac{l}{c_{1}}-\frac{l}{c_{2}}}{\sqrt{\frac{l}{c_{1}}}+\sqrt{\frac{l}{c_{2}}}}+|A^{-1}|\delta(|\bar{x}|+|\bar{y}|)\\
 & \leq & \frac{1}{2c_{1}c_{2}}\frac{l(c_{2}-c_{1})}{2\sqrt{\frac{l}{c_{2}}}}+|A^{-1}|\delta(|\bar{x}|+|\bar{y}|)\\
 & \leq & \frac{c_{2}-c_{1}}{4c_{1}}\sqrt{\frac{l}{c_{2}}}+|A^{-1}|\delta(|\bar{x}|+|\bar{y}|).\end{eqnarray*}
It is clear that as $\delta\rightarrow0$, the above formula goes
to zero, so $|\frac{1}{2}(\bar{x}+\bar{y})|^{2}/|l|\rightarrow0$
as $\delta\rightarrow0$ as needed.
\end{proof}
In Lemma \ref{lem:bounds-on-primes} below, we say that $e_{i}$ is
the $i$th \emph{elementary vector} if it is the $i$th column of
the identity matrix. It is also the eigenvector corresponding to the
eigenvalue $a_{k}$ of $A$. 
\begin{lem}
(Length estimates of vectors) \label{lem:bounds-on-primes} Let $h:\mathbb{R}^{n}\rightarrow\mathbb{R}$
and $A\in\mathbb{R}^{n\times n}$ satisfy Assumption \ref{ass:condns-on-model}
for some $\delta>0$. Suppose $h(\bar{x})=h(\bar{y})=l<0$. Suppose
$\theta>0$ is such that $|d_{\bar{x}}-e_{k}|_{\infty}<\theta$ for
some $d_{\bar{x}}$ pointing in the same direction as $\bar{x}$,
and that the same relation holds for $\bar{y}$. 

Then $|\bar{x}|$ and $|\bar{y}|$ are bounded from below and above
by\[
(1-\theta)\sqrt{\frac{l}{(a_{k}-\delta)(1+\theta)^{2}+(n-1)(a_{n}-\delta)\theta^{2}}}\leq|\bar{x}|,|\bar{y}|,\]
\[
|\bar{x}|,|\bar{y}|\leq\sqrt{[(1+\theta)^{2}+(n-1)\theta^{2}]\frac{l}{(a_{k}+\delta)(1-\theta)^{2}+(n-1)(a_{1}+\delta)\theta^{2}}}.\]
\end{lem}
\begin{proof}
Necessarily we must have $k\geq n-m+1$ because if $k<n-m+1$, then
the eigenvalue $a_{k}$ is positive, making the direction $e_{k}$
and its nearby directions directions of ascent. Let $d:=d_{\bar{x}}$.
From $|d-e_{k}|_{\infty}<\theta$, we obtain $1-\theta\leq d_{k}\leq1+\theta$
and $|d_{j}|<\theta$ for $j\neq k$.

For the direction $d$, we find the largest and smallest value for
the norm of $\bar{x}$ if $\mbox{unit}(\bar{x})=\mbox{unit}(d)$.
This is also the largest and smallest possible value of $t$ such
that $h(t\mbox{ unit}(d))=l$. Here, $\mbox{unit}(\cdot)$ maps a
nonzero vector to the unit vector of the same direction. Now\begin{eqnarray*}
h\big(t\mbox{ unit}(d)\big) & \leq & \sum_{i=1}^{n}(a_{i}-\delta)[t\mbox{ unit}(d)_{i}]^{2}\\
 & = & t^{2}\sum_{i=1}^{n}(a_{i}-\delta)\frac{d_{i}^{2}}{\sum_{j=1}^{n}d_{j}^{2}}\\
 & \leq & \frac{t^{2}}{\sum_{j=1}^{n}d_{j}^{2}}\left((a_{k}-\delta)d_{k}^{2}+(a_{n}-\delta)\sum_{j\neq k}d_{j}^{2}\right)\\
 & \leq & \frac{t^{2}}{(1-\theta)^{2}}[(a_{k}-\delta)(1+\theta)^{2}+(n-1)(a_{n}-\delta)\theta^{2}].\end{eqnarray*}
Since $h(t\mbox{ unit}(d))=l$, we have \[
t\leq(1-\theta)\sqrt{\frac{l}{(a_{k}-\delta)(1+\theta)^{2}+(n-1)(a_{n}-\delta)\theta^{2}}}.\]
Next, \begin{eqnarray*}
h(t\mbox{ unit}(d)) & \geq & \sum_{i=1}^{n}(a_{i}+\delta)[t\mbox{ unit}(d)_{i}]^{2}\\
 & = & \frac{t^{2}}{\sum_{j=1}^{n}d_{j}^{2}}\sum_{i=1}^{n}(a_{i}+\delta)d_{i}^{2}\\
 & \geq & \frac{t^{2}}{(1+\theta)^{2}+(n-1)\theta^{2}}[(a_{k}+\delta)(1-\theta)^{2}+(n-1)(a_{1}+\delta)\theta^{2}].\end{eqnarray*}
Again, since $h(t\mbox{ unit}(d))=l$, we have \[
t\geq\sqrt{[(1+\theta)^{2}+(n-1)\theta^{2}]\frac{l}{(a_{k}+\delta)(1-\theta)^{2}+(n-1)(a_{1}+\delta)\theta^{2}}}.\]

\end{proof}
Here are some lemmas on the completion of orthogonal matrices. In
Lemmas \ref{lem:first-matrix-approx} and \ref{lem:matrix-approx},
let $|\cdot|$ denote a norm for matrices (which need not be a matrix
norm).
\begin{lem}
(Completion to orthogonal matrix) \label{lem:first-matrix-approx}Let
$E_{k}\in\mathbb{R}^{n\times k}$ be the first $k$ columns of the
$n\times n$ identity matrix. Then for all $\epsilon>0$, there exists
$\delta>0$ such that if $V_{k}\in\mathbb{R}^{n\times k}$ has orthonormal
columns and $|V_{k}-E_{k}|<\delta$ , then $V_{k}$ can be completed
to an orthogonal matrix $V_{n}\in\mathbb{R}^{n\times n}$ such that
$|I-V_{n}|<\epsilon$.
\end{lem}
The above lemma is an easy consequence of the following result.
\begin{lem}
(Finding orthogonal vector) \label{lem:matrix-approx}For all $\epsilon>0$,
there exists a $\delta>0$ such that if $V_{k}\in\mathbb{R}^{n\times k}$
has orthonormal columns and $|V_{k}-E_{k}|<\delta$ , then there is
a vector $v_{k+1}\in\mathbb{R}^{n}$ such that $|v_{k+1}|_{2}=1$
and is orthogonal to all columns of $V_{k}$, and the concatenation
$V_{k+1}:=[V_{k},v_{k+1}]$ satisfies $|V_{k+1}-E_{k+1}|<\epsilon$.\end{lem}
\begin{proof}
Since all finite dimensional norms are equivalent, we can assume that
the norm $|\cdot|$ on $\mathbb{R}^{n\times k}$, $\mathbb{R}^{n\times(k+1)}$
is the $\infty$-norm for vectors, that is $|M|=\max_{i,j}|M(i,j)|$.
Suppose $|V_{k}-E_{k}|<\delta$. Then $|V_{k}(i,j)|<\delta$ if $i\neq j$
and $|V_{k}(i,i)-1|<\delta$. We now construct the vector $v_{k+1}$
using the Gram-Schmidt process. 

The direction of $v_{k+1}$ obtained by the Gram-Schmidt process is:\begin{eqnarray*}
(I-V_{k}V_{k}^{T})e_{k+1} & = & e_{k+1}-V_{k}V_{k}^{T}e_{k+1}\\
 & = & e_{k+1}-\sum_{i=1}^{k}V_{k}(i,k+1)V_{k}(:,i).\end{eqnarray*}
Since $|V_{k}(i,j)|<\delta$ for all $i\neq j$, the sum $\alpha_{k+1}\in\mathbb{R}^{n}$
defined by $\alpha_{k+1}=\sum_{i=1}^{k}V_{k}(i,k+1)V_{k}(:,i)$ has
components obeying the bounds\[
|\alpha_{k+1}(j)|\leq\left\{ \begin{array}{ll}
k\delta^{2} & \mbox{ if }j\geq k+1,\\
(k-1)\delta^{2}+\delta & \mbox{ if }j<k+1.\end{array}\right.\]
Then $v_{k+1}=\mbox{unit}(e_{k+1}+\alpha_{k+1})$. We first analyze
the maximum error in $v_{k+1}(j)$ for $j\neq k+1$. The $2-$norm
of $e_{k+1}+\alpha_{k+1}$ is at least \begin{eqnarray*}
|e_{k+1}+\alpha_{k+1}|_{2} & \geq & |e_{k+1}|_{2}-|\alpha_{k+1}|_{2}\\
 & \geq & 1-\sqrt{\sum_{i=1}^{n}[\alpha_{k+1}(i)]^{2}}\\
 & \geq & 1-\sqrt{(n-k)k\delta^{2}+k\left((k-1)\delta^{2}+\delta\right)}\\
 & \geq & 1-\sqrt{nk\delta^{2}+k\delta}.\end{eqnarray*}
If $\delta<\min\{1,\frac{1}{4(n+1)k}\}$, then $\sqrt{nk\delta^{2}+k\delta}\leq\frac{1}{2}$,
and so $\left|e_{k+1}+\alpha_{k+1}\right|_{2}\geq\frac{1}{2}$. In
this case, the maximum error in $v_{k+1}(j)$ is $2((k-1)\delta^{2}+\delta)$
for $j\neq k$. For $j=k+1$, we note that $|v_{k+1}|_{2}=1$, which
tells us that \begin{eqnarray*}
[v_{k+1}(k+1)]^{2} & = & 1-\sum_{1\leq i\leq n,i\neq k}|v_{k+1}(i)|^{2}\\
 & \geq & 1-4(n-1)[(k-1)\delta^{2}+\delta]^{2}\\
\Rightarrow v_{k+1}(k+1) & \geq & \sqrt{1-4(n-1)[(k-1)\delta^{2}+\delta]^{2}}\\
 & = & \sqrt{1-4(n-1)\delta^{2}[(k-1)\delta+1]^{2}}.\end{eqnarray*}
If $\delta<\min\{1,\frac{1}{2k\sqrt{n-1}}\}$, then $4(n-1)\delta^{2}[(k-1)\delta+1]^{2}<1$,
which gives \begin{eqnarray*}
v_{k+1}(k+1) & \geq & \sqrt{1-4(n-1)\delta^{2}[(k-1)\delta+1]^{2}}\\
 & \geq & 1-4(n-1)\delta^{2}[(k-1)\delta+1]^{2}\\
|v_{k+1}(k+1)-1| & \leq & 4(n-1)\delta^{2}[(k-1)\delta+1]^{2}.\end{eqnarray*}
If $\delta<\min\{\frac{\epsilon}{2k},\frac{1}{2k}\sqrt{\frac{\epsilon}{n-1}}\}$,
then \begin{eqnarray*}
|[V_{k},v_{k+1}]-E_{k+1}| & \leq & \max\{2[(k-1)\delta^{2}+\delta],4(n-1)\delta^{2}[(k-1)\delta+1]^{2}\}\\
 & < & \epsilon,\end{eqnarray*}
 and we are done.
\end{proof}
The next result shows that we get a closer estimate to the critical
value after each iteration of step 2 in Algorithm \ref{alg:fast-local-method}. 
\begin{lem}
(Lower bound on critical value) \label{lem:bdd-on-l-(i+1)}Let $\delta>0$
be sufficiently small, and suppose $h:\mathbb{R}^{n}\rightarrow\mathbb{R}$
satisfies Assumption \ref{ass:condns-on-model}. Let $S_{\bar{z},V}:=\{\bar{z}+Vw\mid w\in\mathbb{R}^{n-m}\}$,
and $V\in\mathbb{R}^{n\times(n-m)}$ be such that $V$ has orthonormal
columns, with $|\bar{z}-\mathbf{0}|<\delta$ and $|V-E_{n-m}|<\delta$,
where $E_{n-m}\in\mathbb{R}^{n\times(n-m)}$ is the first $n-m$ columns
of the identity matrix. Then\[
-\frac{1}{2}|A-\delta I|\left(1+|[V^{T}(A-\delta I)V]^{-1}||V^{T}||A-\delta I|\right)^{2}|\bar{z}|^{2}\leq\min_{s\in S_{\bar{z},V}\cap\mathbb{B}}h(s).\]
\end{lem}
\begin{proof}
We find a lower bound for the smallest value of $h$ on $S_{\bar{z},V}$.
The function $h_{\bar{z},V}:\mathbb{R}^{n-m}\rightarrow\mathbb{R}$
defined by $h_{\bar{z},V}(w):=h(\bar{z}+Vw)$ satisfies\begin{eqnarray*}
h_{\bar{z},V}(w) & = & h(\bar{z}+Vw)\\
 & \geq & \frac{1}{2}(\bar{z}+Vw)^{T}(A-\delta I)(\bar{z}+Vw).\end{eqnarray*}
Let us denote $h_{\bar{z},V,\min}:\mathbb{R}^{n-m}\rightarrow\mathbb{R}$
by $h_{\bar{z},V,\min}(w)=\frac{1}{2}(\bar{z}+Vw)^{T}(A-\delta I)(\bar{z}+Vw)$.
The Hessian of $h_{\bar{z},V,\min}$ is $V^{T}(A-\delta I)V$, which
tells us that $h_{\bar{z},V,\min}$ is strictly convex. Therefore,
we seek to find the minimizer of $h_{\bar{z},V,\min}$.

The minimizing value of $w$, which we denote as $\bar{w}_{\min}$,
satisfies $\nabla h_{\bar{z},V,\min}(\bar{w}_{\min})=\mathbf{0}$.
This gives us \begin{eqnarray*}
V^{T}(A-\delta I)\bar{z}+V^{T}(A-\delta I)V\bar{w}_{\min} & = & \mathbf{0}\\
\Rightarrow\bar{w}_{\min} & = & -[V^{T}(A-\delta I)V]^{-1}V^{T}(A-\delta I)\bar{z}.\end{eqnarray*}
An easy bound on $\left|\bar{w}_{\min}\right|$ is $\left|\bar{w}_{\min}\right|\leq|[V^{T}(A-\delta I)V]^{-1}||V^{T}||A-\delta I||\bar{z}|$.
 So $h_{\bar{z},V,\min}$ is bounded from below by \begin{eqnarray}
\min_{w}h_{\bar{z},V,\min}(w) & = & \min_{w}\frac{1}{2}(\bar{z}+Vw)^{T}(A-\delta I)(\bar{z}+Vw)\nonumber \\
 & = & \frac{1}{2}(\bar{z}+V\bar{w}_{\min})^{T}(A-\delta I)(\bar{z}+V\bar{w}_{\min})\nonumber \\
 & \geq & -\frac{1}{2}|\bar{z}+V\bar{w}_{\min}||A-\delta I||\bar{z}+V\bar{w}_{\min}|\nonumber \\
 & = & -\frac{1}{2}|A-\delta I||\bar{z}+V\bar{w}_{\min}|^{2}\nonumber \\
 & \geq & -\frac{1}{2}|A-\delta I|[|\bar{z}|+\left|V\bar{w}_{\min}\right|]^{2}\nonumber \\
 & = & -\frac{1}{2}|A-\delta I|[|\bar{z}|+|\bar{w}_{\min}|]^{2}\nonumber \\
 & \geq & -\frac{1}{2}|A-\delta I|\left(|\bar{z}|+\left|[V^{T}(A-\delta I)V]^{-1}\right||V^{T}||A-\delta I||\bar{z}|\right)^{2}\nonumber \\
 & = & -\frac{1}{2}|A-\delta I|\left(1+\left|[V^{T}(A-\delta I)V]^{-1}\right||V^{T}||A-\delta I|\right)^{2}|\bar{z}|^{2}.\label{eq:bdd-below-h-subspace}\end{eqnarray}

\end{proof}
We shall prove Lemma \ref{lem:on-S-prime} about the approximation
of the eigenvectors corresponding to the smallest eigenvalues. This
lemma analyzes Algorithm \ref{alg:find-S-prime}. We clarify our notation.
In the case of an exact quadratic, Algorithm \ref{alg:fast-local-method}
first finds $e_{n-m+1}$, then invokes Algorithm \ref{alg:find-S-prime}
to find the eigenvector $e_{n}$, followed by $e_{n-1}$, $e_{n-2}$
and so on, all the way to $e_{n-m+2}$. 

We define $I_{k}$ and $I_{k}^{\perp}$ as subsets of $\{1,\dots,n\}$
by \begin{eqnarray*}
I_{k} & := & \{n-m+1\}\cup\{n-k+2,n-k+3,\dots,n\}\\
I_{k}^{\perp} & := & \{1,\dots,n\}\backslash I_{k}.\end{eqnarray*}
Next, we define $E_{k}^{\prime}$ and $E_{k}^{\perp}$. The matrix
$E_{k}^{\prime}\in\mathbb{R}^{n\times k}$ has the $k$ columns $e_{n-m+1}$,
$e_{n}$, $e_{n-1}$, ..., $e_{n-k+2}$, while the matrix $E_{k}^{\perp}\in\mathbb{R}^{n\times(n-k)}$
contains all the other columns in the $n\times n$ identity matrix.
The columns of $E_{k}^{\prime}$ and $E_{k}^{\perp}$ are chosen from
the $n\times n$ identity matrix from the index sets $I_{k}$ and
$I_{k}^{\perp}$ respectively.

We will need to analyze the eigenvalues of $(V_{k}^{\perp})^{T}(A\pm\delta I)V_{k}^{\perp}$
in the proof of Lemma \ref{lem:on-S-prime}, where $|V_{k}^{\perp}-E_{k}^{\perp}|$
is small. Note that the matrix $(E_{k}^{\perp})^{T}(A\pm\delta I)E_{k}^{\perp}$
is principle minor of $A\pm\delta I$, and its eigenvalues are the
eigenvalues of $A\pm\delta I$ chosen according to the index set $I_{k}^{\perp}$.
Furthermore, \begin{eqnarray*}
 &  & |(V_{k}^{\perp})^{T}(A\pm\delta I)V_{k}^{\perp}-(E_{k}^{\perp})^{T}(A\pm\delta I)E_{k}^{\perp}|\\
 & \leq & |(V_{k}^{\perp})^{T}(A\pm\delta I)V_{k}^{\perp}-(V_{k}^{\perp})^{T}(A\pm\delta I)E_{k}^{\perp}|\\
 &  & \quad+|(V_{k}^{\perp})^{T}(A\pm\delta I)E_{k}^{\perp}-(E_{k}^{\perp})^{T}(A\pm\delta I)E_{k}^{\perp}|\\
 & \leq & |(V_{k}^{\perp})^{T}(A\pm\delta I)||V_{k}^{\perp}-E_{k}^{\perp}|\\
 &  & \quad+|(V_{k}^{\perp})^{T}-(E_{k}^{\perp})^{T}||(A\pm\delta I)E_{k}^{\perp}|\end{eqnarray*}
It is clear that as $|V_{k}^{\perp}-E_{k}^{\perp}|\to0$, $|(V_{k}^{\perp})^{T}(A\pm\delta I)V_{k}^{\perp}-(E_{k}^{\perp})^{T}(A\pm\delta I)E_{k}^{\perp}|\to0$.
The eigenvalues of a matrix varies continuously with respect to the
entries when the eigenvalues are distinct, so we shall let $\hat{a}_{i}$
denote the eigenvalue of $(V_{k}^{\perp})^{T}(A+\delta I)V_{k}^{\perp}$
that is closest to $a_{i}$, and $\tilde{a}_{i}$ denote the eigenvalue
of $(V_{k}^{\perp})^{T}(A-\delta I)V_{k}^{\perp}$ that is closest
to $a_{i}$. 
\begin{lem}
(Estimates of eigenvectors to negative eigenvalues) \label{lem:on-S-prime}Let
$h:\mathbb{R}^{n}\to\mathbb{R}$. Given a fixed $l<0$ sufficiently
close to $0$, let $p$ be the closest point to $\bar{z}$ in the
set $\lev_{\leq l}h\cap S_{\bar{z},V_{k}^{\perp}}$, where $S_{\bar{z},V_{k}^{\perp}}:=\{\bar{z}+V_{k}^{\perp}w\mid w\in\mathbb{R}^{n-k}\}$
and $V_{k}^{\perp}\in\mathbb{R}^{n\times(n-k)}$. Then for all $\epsilon>0$,
there exists $\delta>0$ such that if 
\begin{enumerate}
\item $h:\mathbb{R}^{n}\rightarrow\mathbb{R}$ satisfies Assumption \ref{ass:condns-on-model},
\item $|\bar{z}|<\delta$ and
\item $|V_{k}^{\perp}-E_{k}^{\perp}|<\delta$, where $V_{k}^{\perp}$ has
orthogonal columns,
\end{enumerate}
then $|\unitt(p-\bar{z})-e_{n-k+1}|<\epsilon$. As a consequence,
$|V_{m}^{\perp}-E_{m}^{\perp}|\to0$ as $\delta\to0$.\end{lem}
\begin{proof}
The first step is to find an upper bound on the distance between $\bar{z}$
and $p$. The upper bound is obtained from looking at the closest
distance between $\bar{z}$ and $\lev_{\leq l}h_{\max}\cap S_{\bar{z},V_{k}^{\perp}}$.

We look at the function $h_{\bar{z},V_{k}^{\perp},\max}:\mathbb{R}^{n-k}\rightarrow\mathbb{R}$
defined by $h_{\bar{z},V_{k}^{\perp},\max}(w):=h_{\max}(\bar{z}+V_{k}^{\perp}w)$.
We have\begin{eqnarray*}
h_{\bar{z},V_{k}^{\perp},\max}(w) & = & h_{\bar{z},V_{k}^{\perp},\max}(\bar{z}+V_{k}^{\perp}w)\\
 & = & \frac{1}{2}(\bar{z}+V_{k}^{\perp}w)^{T}(A+\delta I)(\bar{z}+V_{k}^{\perp}w)\\
\Rightarrow\nabla h_{\bar{z},V_{k}^{\perp},\max}(w) & = & (V_{k}^{\perp})^{T}(A+\delta I)\bar{z}+(V_{k}^{\perp})^{T}(A+\delta I)V_{k}^{\perp}w.\end{eqnarray*}
The critical point of $h_{\bar{z},V_{k}^{\perp},\max}$ is thus $\bar{w}_{\max}:=-[(V_{k}^{\perp})^{T}(A+\delta I)V_{k}^{\perp}]^{-1}(V_{k}^{\perp})^{T}(A+\delta I)\bar{z}$.
The critical value corresponding to this is $h_{\bar{z},V_{k}^{\perp},\max}(\bar{w}_{\max})$.
An upper bound for this critical value is $\frac{1}{2}|\bar{z}+V_{k}^{\perp}\bar{w}|\left|A+\delta I\right||\bar{z}+V_{k}^{\perp}\bar{w}|$,
which is in turn bounded by:\begin{eqnarray*}
 &  & \frac{1}{2}|\bar{z}+V_{k}^{\perp}\bar{w}||A+\delta I||\bar{z}+V_{k}^{\perp}\bar{w}|\\
 & = & \frac{1}{2}\left|A+\delta I\right||\bar{z}+V_{k}^{\perp}\bar{w}|^{2}\\
 & \leq & \frac{1}{2}\left|A+\delta I\right|\left(1+\left|[(V_{k}^{\perp})^{T}(A+\delta I)V_{k}^{\perp}]^{-1}\right||(V_{k}^{\perp})^{T}|\left|A+\delta I\right|\right)^{2}|\bar{z}|^{2}.\\
 &  & \mbox{(following calculations similar to that of \eqref{eq:bdd-below-h-subspace}).}\end{eqnarray*}
Then an upper bound of the distance $d(\bar{z},\lev_{\leq l}h_{\max}\cap S_{\bar{z},V_{k}^{\perp}})$
can be calculated by:

\begin{eqnarray*}
 &  & d(\bar{z},S_{\bar{z},V_{k}^{\perp}}\cap\lev_{\leq l}h_{\max})\\
 & \leq & \left|\bar{z}-\left(\bar{z}+V_{k}^{\perp}\bar{w}_{\max}\right)\right|+d(\bar{z}+V_{k}^{\perp}\bar{w}_{\max},S_{\bar{z},V_{k}^{\perp}}\cap\lev_{\leq l}h_{\max})\\
 & \leq & \underbrace{|\bar{w}_{\max}|+\sqrt{\frac{l-\frac{1}{2}\left|A+\delta I\right|\left(1+\left|[(V_{k}^{\perp})^{T}(A+\delta I)V_{k}^{\perp}]^{-1}\right||(V_{k}^{\perp})^{T}|\left|A+\delta I\right|\right)^{2}|\bar{z}|^{2}}{\hat{a}_{n-k+1}+\delta}}.}_{\beta}\end{eqnarray*}

The extra term $-\frac{1}{2}\left|A+\delta I\right|\cdots|\bar{z}|^{2}$
compensates for the fact that the critical value of $h_{\bar{z},V_{k}^{\perp},\max}$
is not necessarily zero. To simplify notation, let $\beta$ be the
right hand side of the above formula as marked.

We now figure the possible intersection between $\mathbb{B}(\bar{z},\beta)$
and $S_{\bar{z},V_{k}^{\perp}}\cap\lev_{\leq l}h$. Again, since $S_{\bar{z},V_{k}^{\perp}}\cap\lev_{\leq l}h\subset S_{\bar{z},V_{k}^{\perp}}\cap\lev_{\leq l}h_{\min}$,
we look at the intersection of $\mathbb{B}(\bar{z},\beta)$ and $S_{\bar{z},V_{k}^{\perp}}\cap\lev_{\leq l}h_{\min}$.
We find the critical point of $h_{\bar{z},V_{k}^{\perp},\min}:\mathbb{R}^{n-k}\rightarrow\mathbb{R}$
defined by $h_{\bar{z},V_{k}^{\perp},\min}(w):=\frac{1}{2}(\bar{z}+V_{k}^{\perp}w)^{T}(A-\delta I)(\bar{z}+V_{k}^{\perp}w)$.
The gradient of $h_{\bar{z},V_{k}^{\perp},\min}$ can be found to
be \[
\nabla h_{\bar{z},V_{k}^{\perp},\min}(w)=(V_{k}^{\perp})^{T}(A-\delta I)\bar{z}+(V_{k}^{\perp})^{T}(A-\delta I)V_{k}^{\perp}w.\]
Once again, the critical point is $\bar{w}_{\min}=[(V_{k}^{\perp})^{T}(A-\delta I)V_{k}^{\perp}]^{-1}(V_{k}^{\perp})^{T}(A-\delta I)\bar{z}$.
So $\mathbb{B}(\bar{z},\beta)\subset\mathbb{B}(\bar{z}+V_{k}^{\perp}\bar{w}_{\min},\beta+|\bar{w}_{\min}|)$.

Consider $p\in\mathbb{B}(\bar{z}+V_{k}^{\perp}\bar{w}_{\min},\beta+|\bar{w}_{\min}|)\cap(S_{\bar{z},V_{k}^{\perp}}\cap\lev_{\leq l}h_{\min})$.
Let us introduce a change of coordinates such that $p=\sum_{i\in I_{k}^{\perp}}\tilde{p}_{i}\tilde{v}_{i}+\bar{w}_{\min}$,
where $\tilde{v}_{i}\in\mathbb{R}^{n-k}$ correspond to the eigenvectors
of $(V_{k}^{\perp})^{T}(A-\delta I)V_{k}^{\perp}$ (in turn corresponding
to the eigenvalues $\tilde{a}_{i}$) and $\tilde{p}_{i}\in\mathbb{R}$
are the multipliers. Then the condition $\bar{z}+V_{k}^{\perp}p\in\mathbb{B}(\bar{z}+V_{k}^{\perp}\bar{w}_{\min},\beta+|\bar{w}_{\min}|)$
and $\bar{z}+V_{k}^{\perp}p\in S_{\bar{z},V_{k}^{\perp}}\cap\lev_{\leq l}h_{\min}$
can be represented as the following constraints respectively:\begin{eqnarray}
\sum_{i\in I_{k}^{\perp}}\tilde{p}_{i}^{2} & \leq & (\beta+|\bar{w}_{\min}|)^{2},\nonumber \\
\sum_{i\in I_{k}^{\perp}}\tilde{a}_{i}\tilde{p}_{i}^{2} & \leq & l+\frac{1}{2}(\bar{z}+V_{k}^{\perp}\bar{w}_{\min})^{T}(A-\delta I)(\bar{z}+V_{k}^{\perp}\bar{w}_{\min}).\label{eq:linear-constraints}\end{eqnarray}
As $\delta\rightarrow0$, the only admissible solution is $\tilde{p}_{n-k+1}=\sqrt{\frac{l}{\tilde{a}_{n-k+1}}}$
and the rest of the $\tilde{p}_{i}$'s are zero. The above constraints
are linear in $\tilde{p}_{i}^{2}$. We consider the minimum possible
value of $\frac{\tilde{p}_{n-k+1}}{\sqrt{\sum\tilde{p}_{i}^{2}}}=\sqrt{\frac{\tilde{p}_{n-k+1}^{2}}{\sum\tilde{p}_{i}^{2}}}$,
which is the dot product between the unit vectors in the direction
of $\tilde{v}_{n-k+1}$ and $p-(\bar{z}+V_{k}^{\perp}\bar{w}_{\min})$.
This is equivalent to the linear fractional program in $\tilde{p}_{i}^{2}$
of minimizing $\frac{\tilde{p}_{n-k+1}^{2}}{\sum\tilde{p}_{i}^{2}}$
subject to the constraints in \eqref{eq:linear-constraints}.

This linear fractional program can be transformed into a linear program
by $q=\frac{1}{\sum\tilde{p}_{i}^{2}}$ and $q_{i}=\frac{\tilde{p}_{i}^{2}}{\sum\tilde{p}_{i}^{2}}$,
which gives:\begin{eqnarray}
\min & q_{n-k+1}\nonumber \\
\mbox{s.t.} & q & \geq\frac{1}{(\beta+|\bar{w}_{\min}|)^{2}},\label{eq:linear-constraint-2.1}\\
 & \sum_{i\in I_{k}^{\perp}}\tilde{a}_{i}q_{i} & \leq\Big[l+\frac{1}{2}(\bar{z}+V_{k}^{\perp}\bar{w}_{\min})^{T}(A-\delta I)(\bar{z}+V_{k}^{\perp}\bar{w}_{\min})\Big]q,\label{eq:linear-constraint-2.2}\\
 & \sum_{i\in I_{k}^{\perp}}q_{i} & =1,\label{eq:linear-constraint-2.3}\\
 & q_{i} & \geq0\mbox{ for all }i\in I_{k}^{\perp}.\nonumber \end{eqnarray}
The constraints of the linear program above gives \begin{eqnarray*}
\sum_{i\in I_{k}^{\perp}}-\tilde{a}_{i}q_{i}+\tilde{a}_{n-k}\sum_{i\in I_{k}^{\perp}}q_{i} & \geq & -\Big[l+\frac{1}{2}(\bar{z}+V_{k}^{\perp}\bar{w}_{\min})^{T}(A-\delta I)(\bar{z}+V_{k}^{\perp}\bar{w}_{\min})\Big]q+\tilde{a}_{n-1}\\
\Rightarrow\sum_{i\in I_{k}^{\perp}}(-\tilde{a}_{i}+\tilde{a}_{n-k})q_{i} & \geq & -\Big[l+\frac{1}{2}(\bar{z}+V_{k}^{\perp}\bar{w}_{\min})^{T}(A-\delta I)(\bar{z}+V_{k}^{\perp}\bar{w}_{\min})\Big]q+\tilde{a}_{n-1}\\
 & \geq & -\frac{l+\frac{1}{2}(\bar{z}+V_{k}^{\perp}\bar{w}_{\min})^{T}(A-\delta I)(\bar{z}+V_{k}^{\perp}\bar{w}_{\min})}{(\beta+|\bar{w}_{\min}|)^{2}}+\tilde{a}_{n-1}.\end{eqnarray*}
Since only $-\tilde{a}_{n-k+1}+\tilde{a}_{n-k}$ is positive and the
other $-\tilde{a}_{i}+\tilde{a}_{n-k}$ are nonpositive, we have\begin{eqnarray*}
 &  & (-\tilde{a}_{n-k+1}+\tilde{a}_{n-k})q_{n-k+1}\\
 & \geq & \sum_{i\in I_{k}^{\perp}}(-\tilde{a}_{i}+\tilde{a}_{n-k})q_{i}\\
 & \geq & -\frac{l+\frac{1}{2}(\bar{z}+V_{k}^{\perp}\bar{w}_{\min})^{T}(A-\delta I)(\bar{z}+V_{k}^{\perp}\bar{w}_{\min})}{(\beta+|\bar{w}_{\min}|)^{2}}+\tilde{a}_{n-k}\end{eqnarray*}
\[
\Rightarrow q_{n-k+1}\geq\frac{1}{-\tilde{a}_{n-k+1}+\tilde{a}_{n-k}}\left(-\frac{l+\frac{1}{2}(\bar{z}+V_{k}^{\perp}\bar{w}_{\min})^{T}(A-\delta I)(\bar{z}+V_{k}^{\perp}\bar{w}_{\min})}{(\beta+|\bar{w}_{\min}|)^{2}}+\tilde{a}_{n-k}\right).\]
The limit of the right hand side goes to $1$ as $\delta\rightarrow0$,
so this means that $\bar{z}-p$ is close to the direction of the eigenvector
corresponding to the eigenvalue $\tilde{a}_{n-k+1}$ in $(V_{k}^{\perp})^{T}AV_{k}^{\perp}$,
which in turn converges to $e_{n-k+1}$. The proof of this lemma is
complete.

The conclusion that $|V_{m}^{\perp}-E_{m}^{\perp}|\to0$ as $\delta\to0$
follows from the first part of this lemma and Lemma \ref{lem:first-matrix-approx}.
\end{proof}
With these lemmas set up, we are now ready to prove the fast local
convergence of Algorithm \ref{alg:fast-local-method} to the critical
point and critical value. We recall that \emph{Q-linear convergence}
of a sequence of positive numbers $\{\alpha_{i}\}_{i=1}^{\infty}$
converging to zero is defined by $\limsup_{i\to\infty}\frac{\alpha_{i+1}}{\alpha_{i}}<1$,
while \emph{Q-superlinear convergence} is defined by $\lim_{i\to\infty}\frac{\alpha_{i+1}}{\alpha_{i}}=0$.
Next, \emph{R-linear convergence} and \emph{R-superlinear convergence}
of a sequence are defined by being bounded by a Q-linearly convergent
sequence and a Q-superlinearly convergent sequence respectively.
\begin{thm}
(Fast convergence of Algorithm \ref{alg:fast-local-method}) \label{thm:Wrap-up}Suppose
that $f:\mathbb{R}^{n}\to\mathbb{R}$ is $\mathcal{C}^{2}$ and $\mathbf{0}$
is a nondegenerate critical point of $f$ of Morse index $m$, and
$h(\mathbf{0})=0$. 

There is some $R>0$ such that if $0<r<R$, then for $U_{i}=\mathbb{B}(\mathbf{0},r)$
and $l_{0}<0$ (depending on $r$) sufficiently close to $0$, Algorithm
\ref{alg:fast-local-method} converges R-superlinearly to the critical
point $\mathbf{0}$ and Q-superlinearly to the critical value $0$
provided that at each iteration, there exists an optimizing triple
$(S_{i},x_{i},y_{i})$ for which $(x_{i},y_{i})$ is the unique pair
of points $S_{i}\cap(\lev_{\geq l_{i}}f)\cap\mathbb{B}(\mathbf{0},r)$
such that $|x_{i}-y_{i}|=\diam(S_{i}\cap(\lev_{\geq l_{i}}f)\cap\mathbb{B}(\mathbf{0},r))$. 

If a sufficiently good approximate for the lineality space of each
$S_{i}$ is available in step 2 of Algorithm \ref{alg:fast-local-method}
instead, then Algorithm \ref{alg:fast-local-method} converges R-linearly
to the critical point and Q-linearly to the critical value.\end{thm}
\begin{proof}
As a reminder, Q-linear convergence of the critical value is defined
to be $\limsup_{i\to\infty}\frac{|l_{i+1}|}{|l_{i}|}<\infty$, and
Q-superlinear convergence of the critical value is defined to be $\lim_{i\to\infty}\frac{|l_{i+1}|}{|l_{i}|}=0$.
From the Q-linear (Q-superlinear) convergence of $l_{i}$, we obtain
the R-linear (R-superlinear) convergence of the critical point by
observing that $\limsup_{i\to\infty}\frac{|z_{i}|}{\sqrt{|l_{i}|}}<\infty$,
and that $\sqrt{|l_{i}|}$ converges Q-linearly (Q-superlinearly). 

Since $f$ is $\mathcal{C}^{2}$, for all $\delta>0$, we can find
$R>0$, such that \[
|f(x)-x^{T}Ax|<\delta|x|^{2}\mbox{ for all }x\in\mathbb{B}(\mathbf{0},R).\]
The function $f_{R}:\mathbb{R}^{n}\to\mathbb{R}$ defined by $f_{R}(x):=\frac{1}{R^{2}}f(Rx)$
satisfies Assumption \ref{ass:condns-on-model} with $A:=\nabla^{2}f_{R}(x)=\nabla^{2}f(x)$. 

We want to show that if $\delta>0$ is sufficiently small, then for
all $l<0$ sufficiently small, a step in Algorithm \ref{alg:fast-local-method}
gives good convergence to the critical value. Given an iterate $l_{i}$,
the next iterate $l_{i+1}$ is \[
\min_{x\in S_{z_{i},V_{m}^{\perp}}\cap\mathbb{B}(\mathbf{0},R)}f(x)=\min_{x\in S_{\frac{1}{R}z_{i},V_{m}^{\perp}}\cap\mathbb{B}}R^{2}f_{R}(x),\]
where $V_{m}^{\perp}$, which approximates the first $n-m$ eigenvectors,
is defined before Lemma \ref{lem:on-S-prime}. 

We seek to find $\frac{|l_{i+1}|}{|l_{i}|}$. The value of $l_{i+1}$
depends on how well the last $m$ eigenvectors are approximated, and
how well the critical point is estimated, which in turn depends on
$\delta$. The ratio $\frac{|l_{i+1}|}{|l_{i}|}$ is bounded from
above by\[
-\frac{1}{2}|A-\delta I|\left(1+\left|[(V_{m}^{\perp})^{T}(A-\delta I)V_{m}^{\perp}]^{-1}\right||(V_{m}^{\perp})^{T}||A-\delta I|\right)^{2}|z_{i}|^{2}/l_{i},\]
which converges to $0$ as $\delta\to0$ by Lemmas \ref{lem:min-maximizer-ppties},
\ref{lem:first-matrix-approx}, \ref{lem:bdd-on-l-(i+1)} and \ref{lem:on-S-prime}. 

The conclusion in the second part of the theorem follows a similar
analysis.
\end{proof}

\section{Conclusion and conjectures}

In this paper, we present a strategy to find saddle points of general
Morse index, extending the algorithms for finding critical points
of mountain pass type as was done in \cite{LP08}. Algorithms \ref{alg:fast-local-method}
and \ref{alg:find-S-prime} may not be easily implementable, especially
when $m$ is large. However, Algorithm \ref{alg:find-S-prime} can
be performed only as needed in a practical implementation. It is hoped
that this strategy can augment current methods for finding saddle
points, and can serve as a foundation for further research on effective
methods of finding saddle points.

Here are some conjectures:
\begin{itemize}
\item How do the algorithms presented fare in real problems? Are there difficulties
in the infinite dimensional case when implementing Algorithm \ref{alg:fast-local-method}?
\item Are there ways to integrate Algorithms \ref{alg:fast-local-method}
and \ref{alg:find-S-prime} to give a better algorithm?
\item Are there better algorithms than Algorithm \ref{alg:find-S-prime}
to approximate $S_{i}$?
\item Can the uniqueness assumption in Theorem \ref{thm:Wrap-up} be lifted?
If not, how does it affect the design of algorithms?\end{itemize}

\begin{acknowledgement*}
I thank the Fields Institute in Toronto, where much of this paper
was written. They have provided a wonderful environment for working
on this paper.\end{acknowledgement*}

\end{document}